\documentclass[11pt]{article}
\usepackage[utf8]{inputenc}
\usepackage[dvipsnames]{xcolor}
\usepackage{comment}
\usepackage{amsmath, amsthm, amssymb}
\usepackage[colorlinks,citecolor=blue,linkcolor=BrickRed]{hyperref}\usepackage[colorlinks,citecolor=blue,linkcolor=BrickRed]{hyperref}
\usepackage{graphicx}
\usepackage{mathtools}
\usepackage{soul}
\usepackage{thmtools} 
\usepackage{thm-restate}
\usepackage{geometry}
\usepackage{cleveref}
\usepackage{enumitem}

\newtheorem{thm}{Theorem}[section]
\newtheorem{conj}[thm]{Conjecture}
\newtheorem{cor}[thm]{Corollary}
\newtheorem{lemma}[thm]{Lemma}

\newtheorem{prop}[thm]{Proposition}

\theoremstyle{definition}
\newtheorem{defn}[thm]{Definition}
\newtheorem{rem}[thm]{Remark}

\newtheorem{fact}[thm]{Fact}

\definecolor{AfonsoBlue}{RGB}{30,65,123}

\DeclareMathOperator{\poly}{poly}
\DeclareMathOperator{\fr}{F}

\DeclareMathOperator{\vol}{vol}
\newcommand{\R}{\mathbb R}

\newcommand{\Prob}{\mathbb P}
\newcommand{\E}{\mathbb E}
\newcommand{\EE}{\mathbb E}

\newcommand{\edball}{\mathbb{B}_2^d}
\newcommand{\pdball}{\mathbb{B}_p^d}

\newcommand{\rv}[1]{\mathbf{#1}}
\newcommand{\abs}[1]{\vert#1\vert}
\newcommand{\Abs}[1]{\left\vert#1\right\vert}
\newcommand{\norm}[1]{\left\lVert#1\right\rVert}

\newcommand{\norminp}[1]
{\left\lVert#1\right\rVert_{\mathcal{I}_p}}

\newcommand{\norminx}[2]
{\left\lVert#1\right\rVert_{\mathcal{I}_{#2}}}
\newcommand{\pvar}[1]{\sigma_{(#1,\mathcal{I}_p)}}
\newcommand{\evar}[1]{\sigma_{(#1, \mathcal{I}_2)}}
\newcommand{\typetwopvar}{\sigma_{T_2 \mathcal{I}_p}}
\newcommand{\normf}[1]{\left\lVert#1\right\rVert_{\fr}}

\DeclareMathOperator{\sym}{sym}
\newcommand{\eps}{\varepsilon}
\newcommand{\md}{\mathbf{d}}
\newcommand{\op}{\mathsf{op}}
\newcommand{\Dec}{\mathsf{Dec}}
\newcommand{\calC}{\mathcal{C}}
\newcommand{\I}{\mathcal{I}}
\newcommand{\p}{\mathbb{P}}
\newcommand{\polylog}{\mathsf{polylog}}
\renewcommand{\H}{\mathcal{H}}
\newcommand{\NNN}{\mathcal{N}}
\newcommand{\CCC}{\mathcal{C}}
\DeclareMathOperator{\Tr}{Tr}
\newcommand{\cleq}[1]{\underset{#1}{\lesssim}}
\newcommand{\ceq}[1]{\underset{#1}{\asymp}}
\newcommand{\cgeq}[1]{\underset{#1}{\gtrsim}}
\newcommand{\diag}{\mathsf{diag}}

\title{A Geometric Perspective on the Injective Norm of Sums of Random Tensors}

\author{
Afonso S.\ Bandeira\thanks{Department of Mathematics, ETH Z\"urich. \texttt{bandeira@math.ethz.ch}.} 
 \and 
Sivakanth Gopi\thanks{Microsoft Research, Redmond. \texttt{sivakanth.gopi@microsoft.com}.}
 \and
Haotian Jiang\thanks{Department of Computer Science, University of Chicago. \texttt{jhtdavid@uchicago.edu}. Part of the work was done while being a Postdoc Researcher in the Algorithms Group at Microsoft Research Redmond.}
\and
Kevin Lucca\thanks{Department of Mathematics, ETH Z\"urich. \texttt{kevin.lucca@ifor.math.ethz.ch}.}
\and
Thomas Rothvoss\thanks{Department of Mathematics and
Paul G. Allen School of CSE, University of Washington. \texttt{rothvoss@uw.edu}. Supported by NSF grant 2318620 ``\emph{AF: SMALL: The Geometry of Integer Programming and Lattices}''.}
}

\date{}

\begin{document}

\maketitle

\begin{abstract}
Matrix concentration inequalities, intimately connected to the Non-Commutative Khintchine inequality, have been an important tool in both applied and pure mathematics. We study tensor versions of these inequalities, and establish non-asymptotic inequalities for the $\ell_p$ injective norm of random tensors with correlated entries. In certain regimes of $p$ and the tensor order, our tensor concentration inequalities are nearly optimal in their dimension dependencies. We illustrate our result with  applications to problems including structured models of random tensors and matrices, tensor PCA, and connections to lower bounds in coding theory.

Our techniques are based on covering number estimates as opposed to operator theoretic tools, which also provide a geometric proof of a weaker version of the Non-Commutative Khintchine inequality, motivated by a question of Talagrand. 
\end{abstract}

\newpage 

\tableofcontents

\newpage

\section{Introduction}

\emph{The following introduction is mainly targeted to a Mathematics audience. An introduction and technical overview focusing more on the Theoretical Computer Science motivations behind this work is provided in Appendix~\ref{sec:introTCS}.}\medskip

The main object of study of this paper is the \emph{$\ell_p$ injective norm} of a jointly Gaussian random $d$-dimensional, $r$-order, tensor $T\in(\R^d)^{\otimes r}$ with an arbitrary entrywise covariance structure. We denote the $\ell_p$ injective norm\footnote{It is also natural to consider a symmetric version of the $\ell_p$ injective norm where the vectors $x_1, \cdots, x_r$ are forced to be the same, i.e., $\norm{T}_{\sym(\mathcal{I}_p)} \coloneqq \sup_{x \in \pdball} \abs{ T[x, \ldots, x]}$. For symmetric tensors, these two definitions are equivalent up to a factor depending only on the order $r$  (see \Cref{prop:symnormequiv}).} of $T$ by
    \begin{equation}\label{eq:injnormdef}
    \norminp{T} \coloneqq \sup_{\norm{x_1}_p, \ldots, \norm{x_r}_p \leq 1} \langle T, x_1 \otimes \ldots \otimes x_r \rangle,  \end{equation}
where $d$ and $r\geq 2$ are integers and $p\geq2$.

This problem has been well studied for the specific case of iid standard Gaussian entries and $p=2$, in that situation it is known that the expected injective norm is of the order of $C_r\sqrt{d}$, for a constant $C_r$ depending on $r$. The order of the constant $C_r$ was shown to be bounded from above by $C\sqrt{r\log(r)}$ in~\cite{TS14,NDT10} and sharper estimates for the precise constant have been computed recently in~\cite{DM24}.

Any tensor $T\in(\R^d)^{\otimes r}$ with jointly Gaussian entries can be written in a Gaussian series form as
\begin{equation}\label{eq:Tasgaussianseries}
T = \sum_{k=1}^n g_kT_k,
\end{equation}
where $g_k\sim \NNN(0,1)$ are iid and $T_i$ are deterministic tensors.
In the specific euclidean matrix case ($p=2$ and $r=2$),~\eqref{eq:injnormdef} corresponds to the operator norm of a random Gaussian matrix. An important consequence of the celebrated \emph{Non-Communitative Khintchine inequalities} of Lust-Piquard and Pisier (see~\cite{LP91,Lus86}), which was noted in~\cite{Rud96o}, states that, for $X = \sum_{k=1}^n g_kA_k$,
\begin{equation}\label{eq:NCK}
\EE \Big\| \sum_{k=1}^n g_kA_k \Big\|_{\mathrm{op}} 
\lesssim \sqrt{\log d} \, 
\max\bigg\{\Big\| \sum_{k=1}^n A_k^TA_k \Big\|_{\mathrm{op}}^{\frac12}  ,\Big\| \sum_{k=1}^n A_kA_k^T \Big\|_{\mathrm{op}}^{\frac12}  \bigg\}. 
\end{equation}
Notably,~\eqref{eq:NCK} is tight in two senses: the right-hand side is also a lower bound without the logarithmic factor and, moreover, the logarithmic factor is required for some choices of matrices. We point the reader to~\cite{BBvH23} for a recent line of work trying to better understand when the logarithmic factor is not required.

However, the question of estimating the injective norm of~\eqref{eq:Tasgaussianseries} becomes elusive when either $r \neq 2$ or $p \neq 2$, which are natural extensions of~\eqref{eq:NCK} that appear in coding theory~\cite{Bri16,Gop18} or dispersive partial differential equations~\cite{BDNY24}. 

The main reason for this bottleneck is the fact that all known proofs of~\eqref{eq:NCK} use operator-theoretic tools,\footnote{The study of Schatten-$p$ norms (traces of powers) and the fact that they approximate the spectral norm well for large $p$ are key ingredients in the proof.} which are thus far unavailable when either $r \neq 2$ or $p \neq 2$. We mention two notable results in this direction that do not assume independent entries, the first one is by Rudelson \cite{Rud96c}, who proved~\eqref{eq:NCK} (with a logarithmic factor in $n$ rather than $d$) for the specific case when the matrices $A_k$ are symmetric rank-1 using Majorizing Measures. The second result is a volumetric bound by Latała~\cite{Lat06} in the case $p=2$, which was used to derive two-sided moment inequalities for Gaussian chaoses (see \Cref{rem:latalareference}). Rudelson's original argument in \cite{Rud96c} can be adapted to rank-1 tensors (for $p=2$)~\cite{Kevinsthesis}, but it is unclear how to go beyond rank-1, or $p > 2$. In fact, there is no known geometric argument to establish~\eqref{eq:NCK}, even for the euclidean matrix case, when the summands are high rank. This is another core motivation of this paper, and in fact, a question in Talagrand's book~\cite{Tal14} (section 16.10).

\subsection{Main Results}

Our main contribution is a geometric argument, involving covering numbers and Gaussian process theory, to bound the expected value of the $\ell_p$ injective norm of~\eqref{eq:Tasgaussianseries} for a full range of $r\geq 2$ and $ 2 \leq p <\infty$, where operator theoretic tools are unavailable. 

For general tensors it is more convenient to formulate the upper bound in terms of the so called Type-2 constant of the corresponding Banach space.

\begin{defn}[Type-2 constant for $\ell_p$ injective norm of order $r$ tensors]\label{def:type2constant}
Given $r\geq 2$ an integer and $2 \leq p < \infty$ we define $\CCC_{r,p}(d)$ the Type-2 constant of the $\ell_p$ injective norm of order $r$ tensors to be the minimal number, such that for all positive integers $n$ and tensors $T_1, \ldots, T_n \in (\R^d)^{\otimes r}$ we have
\begin{equation}\label{eq:type2def}
    \E \Big\| \sum_{k=1}^n \eps_k T_k\Big\|_{\I_p}^2 \leq \CCC_{r,p}(d)^2 \sum_{k=1}^n \norminp{T_k}^2.
\end{equation}
The $\eps_1, \ldots, \eps_n$ are i.i.d. Rademacher random variables (i.e.: $ \eps_k = 1$ or $\eps_k = -1$ with probability $\frac12$ each). 
\end{defn}

While our main contributions involve Gaussian series of tensors~\eqref{eq:Tasgaussianseries}, a standard argument based on using Jensen's inequality on the magnitude of the Gaussian shows that moments of the norm of a Rademacher series are upper bounded by the ones of the Gaussian series. In particular, they imply the following estimates. 

\begin{thm}\label{cor:boundtype2constant}
Let $r\geq 2$ an integer, $2 \leq p < \infty$, and $\CCC_{r,p}(d)$ the Type-2 constant of the $\ell_p$ injective norm of order $r$ tensors (see Definition~\ref{def:type2constant}) then
\begin{equation}
d^{\frac12 - \frac1p}\cleq{r,p}  \CCC_{r,p}(d) 
\cleq{r,p} d^{\frac12- \frac1{\max\{p,2r\}}}\log d.
\end{equation}
(The expression $f \lesssim_{r,p} g$ says that the inequality $ f \leq C_{r,p} g$ holds for some constant $C_{r,p}$ that only depends on $r$ and $p$.)
\end{thm}

\begin{rem}\label{rem:nckimprovedtype2}
    In the matrix case ($r=2$) it is possible to improve upon the type constant estimate from \Cref{cor:boundtype2constant} by applying Hölder's inequality to the noncommuative Khintchine inequality.\footnote{This would however be a proof using operator theoretic tools (as it would need noncommuative Khintchine inequality).}
    $$ \E \Big\|\sum_{k=1}^n g_k A_k \Big\|_{\I_p} \leq d^{1- \frac 2 p} \E \Big\|\sum_{k=1}^n g_k A_k\Big\|_{\I_p} \lesssim \sqrt{ \log d} \cdot d^{1- \frac 2 p}\sqrt{\sum_{k=1}^n \norminx{A_k}{2}^2}  $$
    Since $\norminx{A_k}{2} \leq \norminp{A_k}$, this inequality implies $\CCC_{2,p}(d) \lesssim \sqrt{ \log d} \cdot d^{1- \frac 2 p}$, which is better than \Cref{cor:boundtype2constant} when $2 \leq p \leq \frac{8}{3}$.
\end{rem}

Note that for $p \geq 2r$ the upper and lower bounds in \Cref{cor:boundtype2constant} match up to a logarithmic factor (in terms of their dependence on $d$). For $r=2$ and $p=2$, one can easily adapt the Non-Commutative Khintchine inequality~\eqref{eq:NCK} to show that $\CCC_{2,2}(d)\lesssim \sqrt{\log(d)}$, meaning that our upper bound is suboptimal by a $d^{\frac14}$ factor (excluding logarithmic factors). Both $d^{\frac12 - \frac1p}$ and $d^{\frac12- \frac1{\max\{p,r\}}}$ would be natural conjectures for the correct polynomial dependency on $d$, and would match Non-Commutative Khintchine in the euclidean matrix case, unfortunately establishing either upper bound appears to be out of reach of our current tools. In comparison a classical $\eps$-net argument would give $d^{\frac12}$ as an upper bound for $\CCC_{r,p}(d)$.

Our estimates are more accurately written in terms of various tensor parameters that provide sharper control than the sum of squared injective norms (the right-hand-side of~\eqref{eq:type2def}). Among other improvements, they correspond to parameters involving norms of sums of squares of matrices (such as in~\eqref{eq:NCK}) rather than sums of squares of norms.\footnote{The sum of squared norms had already appeared in \cite{Tom74,AW02} in the context of random matrices as a variance parameter, which is not as precise as the one given in~\eqref{eq:NCK}, that gained attraction after it appeared in matrix concentration inequalities~\cite{Oli10,Tro10}.} 
Our estimates also imply better bounds for the inhomogenous independent entry model, as well as models of structured random matrices (see \Cref{sec:examples}).
For the sake of exposition, throughout the majority of the paper we focus on symmetric tensors, as the parameters we introduce are simpler to define in this case.

\begin{defn}[$\star_q$ product]\label{defn:introproduct}
    Let $A,B \in (\R^d)^{\otimes r}$ be tensors. For $0 \leq q \leq r$ we define a $d$-dimensional order $2r-2q$ tensor $A \star_q B$ with entries
    $$ (A \star_q B)_{i_1, \ldots, i_{2r-2q}} \coloneqq \sum_{j_1,\ldots, j_q =1}^d A_{i_1, \ldots, i_{r-q},j_1, \ldots, j_q} B_{j_1, \ldots, j_q,i_{r-q+1}, \ldots, i_{2r-2q}}.$$
\end{defn}

Note that in the matrix case ($r=2$), for $A$ and $B$ $d\times d$ symmetric matrices, $\star_q$ corresponds to familiar matrix operations: $A\star_0B=A\otimes B$ corresponds to the Kronecker product (seen as an order $4$ tensor), $A\star_1B=AB$ to the classical matrix product, and $A\star_2B=\Tr(A^{\mathsf{T}}B)$ to the Hilbert-Schmidt inner product.

\begin{defn}[Variance Parameters]\label{defn:intovariances}
    Let $T \in (\R^d)^{\otimes r}$ be a symmetric random jointly Gaussian tensor and $2 \leq p < \infty $. For $0 \leq q \leq r$ we define the parameters
    $$ \pvar{q}^2 \coloneqq \norminp{ \E[ T \star_q T ]}.$$
\end{defn}

In the case $p=r=2$ the parameter $\pvar{1}$ is the variance parameter appearing in~\eqref{eq:NCK}, meanwhile $\pvar{0}$ also frequently appears in the literature\footnote{We defined the weak variance in a different, but equivalent form to the one in~\cite{Tro15}, see also \Cref{lemma:variances}.} and is often referred to as ``weak variance"~\cite{Tro15}. \Cref{lemma:variances} shows
$$ \evar{1}^2 = \Big\| \sum_{k=1}^n A_k^2 \Big\|_{\mathrm{op}} \qquad \textrm{and} \qquad \evar{0}^2 = \sup_{u,v \in \edball} \sum_{k=1}^n (u^\mathsf{T}A_k v)^2 $$
in the matrix case. Formally stated, our main theorem provides the following upper bound:

\begin{thm}[Main Theorem]\label{thm:masterthm}
    Let $T \in (\R^d)^{\otimes r}$ be a symmetric random jointly Gaussian tensor with $\E[T] = 0$ and $2 \leq p < \infty$.
    \begin{equation}\label{eq:intromainthm} 
    d^{\frac1p -\frac12} \E \norminp{T} \cleq{r,p} (\log d) \,  \pvar{1} +  \max_{2 \leq q \leq r} \pvar{q}^{\frac1q} \pvar{0}^{\frac{q-1}{q}},
    \end{equation}
    where the parameters in the right-hand side are those of \Cref{defn:introproduct} and \ref{defn:intovariances}.
\end{thm}

\begin{par}
    For $p=2$ it may be possible to obtain \Cref{thm:masterthm} without the logarithmic factor in front of $\pvar{1}$ by modifying results in~\cite{Lat06}, see \Cref{rem:latalareference}.
    
    We finish the subsection with a few remarks about symmetric tensors. If one does not strive for optimal dependencies on the order $r$ in upper bounds, the symmetry requirement does not lose generality. In fact, by considering symmetric embeddings of tensors, which are tensor-analogues of the hermitian dilation~\eqref{eq:hermitiandilation}, we can reduce the problem to studying a symmetric object with similar norm properties. We will elaborate on this further in later sections (see \Cref{subsection:symbed}). 
\end{par}

\subsection{Notation and Terminology}
\label{subsec:notation}

We use $\E$ to denote the expectation of a random variable and $\Prob$ for the probability of an event. We use $\vol(\cdot)$ for the Lebesgue measure in $\R^d$. We will often write $\rv{g}_1, \ldots, \rv{g}_n$ for independent standard Gaussian scalar random variables and $\rv{b}$ for a random vector distributed according to the normalized Lebesgue measure on some compact set in $\R^d$ (which is usually given from the context). For a vector $v \in \R^d$ we define $ \norm{v}_p \coloneqq (\sum_{i=1}^d \abs{v_i}^p)^{\frac{1}{p}}$ and use $\pdball$ for the unit ball of the norm $\norm{\cdot}_p$. The tensor $v^{\otimes r} \in (\R^{d})^{\otimes r}$ is an order $r$ tensor with entries $(v^{\otimes r})_{i_1, \ldots, i_r} = \prod_{q=1}^r v_{i_q}$. Similarly for $u_1, \ldots , u_r \in \R^d$ we define $(u_1 \otimes \cdots \otimes u_r)_{i_1, \ldots, i_r} = \prod_{q=1}^r (u_q)_{i_q}$. For tensors $A, B \in (\R^d)^{\otimes r}$ their entry-wise dot product is defined by
$$ \langle A, B \rangle \coloneqq \sum_{i_1, \ldots, i_r=1}^d A_{i_1, \ldots, i_r} B_{i_1, \ldots, i_r}.$$
Moreover, for $0 \leq q \leq r$ their generalized product $A \star_q B \in (\R^d)^{\otimes (2r-2q)} $ is given by
$$ (A \star_q B)_{j_1, \ldots, j_{2r-2q}} \coloneqq \sum_{i_1, \ldots, i_q=1}^d A_{j_1, \ldots, j_{r-q}, i_1, \ldots, i_q} B_{i_1, \ldots, i_q, j_{r-q+1}, \ldots, j_{2r-2q}} $$
if $q > 0$ and 
$$ (A \star_0 B)_{j_1, \ldots, j_{2r}} \coloneqq A_{j_1, \ldots, j_{r}} B_{ j_{r+1}, \ldots, j_{2r}}.$$
It should be noted that this operation is indeed bilinear and in the case $q=r$ we have $A \star_r B = \langle A, B \rangle$ and if $q=1$ and $r=2$ then $A \star_1 B$ is standard matrix multiplication. The product $A \odot B \in (\R^d)^{\otimes r}$ denotes the entry-wise product. We will see tensors as multilinear maps on $\R^d$ too, as it can be more convenient to prove statements using this notation:
$$ A[u_1, \ldots, u_r] \coloneqq \sum_{i_1,\ldots, i_r=1 }^d A_{i_1, \ldots, i_r}(u_1)_{i_1} \cdots (u_r)_{i_r} = \langle A, u_1 \otimes \cdots \otimes u_r \rangle $$
This allows contracting tensors with vectors, for $v \in \R^d$ we write $Av^{\otimes 0} \coloneqq A$ and 
\begin{align*} Av^{\otimes q} \coloneqq (Av^{\otimes {q-1}})[v, \cdot, \ldots, \cdot]
\end{align*}
for $q \leq r$, so $Av^{\otimes q}$ becomes an order $r-q$ tensor.
The Frobenius norm is given by $\normf{A} \coloneqq \sqrt{\langle A, A \rangle}$ and the injective $\ell_p$ tensor norm is written as
$$ \norminp{A} \coloneqq \sup_{u_1, \ldots, u_r \in \pdball}A[u_1, \ldots, u_r] .$$ 
We denote the set of all permutations on the set $\{ 1, \ldots, r \}$ by $S_r$. The tensor $A$ is called \emph{symmetric} if for all permutations $\tau \in S_r$ and all $1 \leq i_1, \ldots, i_r \leq d$ we have $T_{i_1, \ldots i_r} = T_{i_{\tau(1)}, \ldots i_{\tau(r)}} $. The symbols $\lesssim, \lesssim_r, \lesssim_{r,p}$ are inequalities that hold up to a factor that are respectively universal constants, constants depending on $r$, constants depending on $r$ and $p$. If both $\lesssim$ and $\gtrsim$ hold, we also use the symbols $\asymp, \asymp_r, \asymp_{r,p}$ in a similar fashion.
\subsection{Overview of Techniques}

\begin{par}
    We give a brief informal presentation of the external machinery used to derive our results, a more detailed discussion can be found in \Cref{section:preliminaries}. If $T \in (\R^d)^{\otimes r}$ is a symmetric jointly Gaussian tensor, the random variable
    $$ \langle T, x_1 \otimes \cdots \otimes x_r \rangle$$
    is also Gaussian and thus the injective norm of $T$ can be seen as the supremum of a Gaussian process. Using symmetry arguments,\footnote{We will use a generalization of the hermitian dilation discussed in Subsection~(\ref{subsection:symbed}), this step only loses a constant factor depending on the order.} it is possible to reduce the problem of estimating~\eqref{eq:injnormdef} to studying the following expected supremum:
    $$ \E \sup_{u \in \pdball} g_u, \quad g_u \coloneqq \langle T, u^{\otimes r} \rangle$$
    As we do not aim for optimal logarithmic factors, we can use Dudley's entropy integral (see \Cref{thm:entropyintegral}) to bound the expected supremum of this Gaussian process
    $$ \E \sup_{u \in \pdball} g_u \lesssim \int_0^\infty \sqrt{\log \mathcal{N} (\pdball,\md, \eps)} \,d\eps. $$
    Here $\mathcal{N} (\pdball,\md, \eps)$ denotes the smallest cardinality of an $\eps$-covering of the set $\pdball$ with respect to the metric
    $$ \md(u,v) = \E [(g_u - g_v)^2]^{\frac12}.$$
    We estimate the covering numbers of $\pdball$ using a volumetric argument. In particular, we will have to estimate the volume of balls with respect to the metric $\md$ that are intersected with ``halved" $\ell_p$ balls, more specifically the ``half" of an $\ell_p$ ball that retains a small $\ell_p$ norm (see \Cref{cor:halvinglpball}). To properly make sense of this, we use an inequality on $\ell_p$ spaces which had been an unpublished result by Ball and Pisier until it was generalized to Schatten classes in \cite{LBC94}:
    $$ \bigg( \frac{\norm{x-y}_p^p + \norm{x+y}_p^p}{2} \bigg)^{2/p} \leq \norm{x}_p^2 + (p-1)\norm{y}_p^2 $$
    holds for $p \geq 2$ and the reverse inequality holds for $1 \leq p \leq 2$. After a sequence of computations we arrive at \Cref{thm:masterthm}. To deduce \Cref{cor:boundtype2constant} from this result, we use comparison inequalities between entry-wise norms on tensors and $\ell_p$ injective norms (see \Cref{thm:boundedmultilinear}), which have been extensively studied in for example \cite{Per81,DS16,HL34}. These inequalities imply
    $$ \normf{T} \cleq{r} d^{\max \{r/p-1/2,0 \}}\norminp{T}$$
    for any tensor $T \in (\R^d)^{\otimes r}$ and $p \geq 2$, which will allow us to transition from the variance parameters to the sum of squared injective norms.
\end{par}

\subsection{Warm-Up: A Crude Non-Commutative Khintchine Inequality }\label{sec:technique}

\begin{par}
    This section serves as an informal summary of the key ideas presented in the paper by focusing on a simplified setting. We consider $p=r=2$ and $n=d$, we have matrices $A_1', \ldots, A_d' \in \R^{d \times d}$ with operator norm at most $1$, so $\norminx{A_k'}{2} \leq 1$ for all $1 \leq k \leq d$. We will prove the (suboptimal) bound
    $$ \E \Big\|\sum_{ k=1}^d \rv{g}_k A_k'\Big\|_{\I_p} \lesssim d^{\frac34}.$$
    Recall that the non-commutative Khintchine inequality gives the better upper bound of $O(\sqrt{d \log d})$ in this setting. Nonetheless, our proof here is purely geometric.

    Without loss of generality, we replace the matrices $A_k'$ by their hermitian dilation, which does not change the operator norm of the matrix 
    and only increases the dimension $d$ by a factor of $2$, so we redefine
    \begin{equation}\label{eq:hermitiandilation}
    A_k \coloneqq \begin{pmatrix}
        0 & A_k' \\
        {A_k'}^\mathsf{T} & 0
    \end{pmatrix}.
    \end{equation}
    Matrices of this form also satisfy $\sup_{\norm{u}_2 \leq 1} u^\mathsf{T}Au = \sup_{\norm{u}_2 \leq 1} \abs{u^\mathsf{T}Au}=\norminx{A}{2}$. (The first equality follows from the fact that the sign of $u^\mathsf{T}Au$ can always be changed by changing the sign of the first half of the coordinates of $u$ and the second equality holds since $A$ is symmetric.) Hence, we want to estimate
    $$ \E \sup_{\norm{u}_2 \leq 1} \sum_{k=1}^d \rv{g}_ku^\mathsf{T} A_ku$$
    using Dudley's entropy integral~(\ref{thm:entropyintegral}). The natural distance on the Gaussian process is given by
    $$\md(u,v) = \sqrt{\sum_{k=1}^d (u^\mathsf{T} A_ku- v^\mathsf{T} A_kv)^2}$$
    Instead of finding the covering numbers $\mathcal{N}(\edball, \md, \eps)$ we find the size of a maximal $\eps$-separated set $D$, which is also an $\eps$-covering of $\edball$ (indeed, if there was a point with distance greater than $\eps$ from $D$, then the set would not be a maximal seperated set). Let $x_1, \ldots, x_m \in \edball$ be a maximal $\eps$-separated set with respect to $\md$, then the balls
    $$ B_\md(x_i, \textstyle \frac{\eps}{2}) \coloneqq \{ x \in \R^d \colon \: \md(x,x_i) < \frac{\eps}{2} \}$$
    are pairwise disjoint for $1 \leq i \leq m$. Namely, since $\md$ is a semi-metric on $\R^d$,  
    $$\eps \leq \md(x_i,x_j) \leq  \md(x_i,x) + \md(x, x_j) < \frac{\eps}{2} + \frac{\eps}{2}  $$
    would otherwise lead to a contradiction for $i \neq j$. We intersect these balls with euclidean half-spheres of radius $R>0$
    $$ G_i \coloneqq B_\md(x_i, \textstyle \frac{\eps}{2}) \cap \left \{ x \in (x_i + R \cdot \edball ) \colon \:  \langle x-x_i, x_i \rangle \leq 0 \right \}.$$
    The sets $G_i$ are also pairwise disjoint. Moreover, the additional condition $\langle x-x_i, x_i \rangle \leq 0$ implies
    $$ \norm{x}_2^2 = \norm{x-x_i}_2^2 + 2\langle x-x_i, x_i \rangle + \norm{x_i}_2^2 \leq R^2 + \norm{x_i}_2^2,$$
    and hence $G_i$ is also contained in the set $ \sqrt{1+R^2} \cdot \edball$. By additivity of the Lebesgue measure we therefore have
    $$ m \cdot \hspace{-1mm} \min_{1 \leq i \leq m} \{  \vol(G_i) \} \leq \sum_{i=1}^m \vol(G_i) \leq \vol \left( \sqrt{1+R^2} \cdot \edball \right).$$
    Since $\abs{D} = m$ is the cardinality of maximal $\eps$-separated set, we can rearrange this inequality to get
    $$ \mathcal{N}(\edball, \md, \eps) \leq \abs{D} \leq   \frac{ (1+R^2)^{\frac{d}{2}} \vol( \edball )}{\min_{1 \leq i \leq m} \{  \vol(G_i) \} }.$$
    It remains to choose an appropriate $R>0$ to make the fraction above small. We choose $R = \min \{ C^{-1}\eps, \sqrt{C^{-1}\eps} \}$ for a sufficiently large chosen constant $C$. To get a lower bound on $\min_{1 \leq i \leq m} \{  \vol(G_i) \}$ we compute the expected distance between a random point $\rv{y} = x_i + R\rv{b}$ uniformly distributed in $x_i + R \cdot \edball$.
    $$  \E [\md(x_i,\rv{y}) ]^2 \leq  \E [\md(x_i,\rv{y})^2 ] = \sum_{k=1}^d \E[( \rv{y}^\mathsf{T} A_k \rv{y}- x_i^\mathsf{T} A_k x_i)^2] = \sum_{k=1}^d \E[( 2R x_i^\mathsf{T} A_k \rv{b}- R^2 \rv{b}^\mathsf{T} A_k \rv{b})^2]  $$
    We skip the precise computation here and do it rigorously in the more general setting in the following sections. For the estimation we mainly rely on the facts $\E[\rv{b}_j^2] \leq \frac{1}{d}$, $\E[\rv{b}_j\rv{b}_{j'}]= 0 $ for $j \neq j'$ and that the matrices $A_k$ have norm at most $1$ and are diagonal-free. The next inequality is valid for the universal constant $C>0$ from before, if it was chosen to be sufficiently large:
    $$ \E [\md(x_i,\rv{y}) ] \leq \frac{C}{8}\max \{ R, R^2\} = \frac{\eps}{8} $$
    Therefore by Markov's inequality we have
    $$ \Prob[ \md(x_i,\rv{y}) \geq \textstyle \frac{\eps}{2}] \leq \frac{1}{4}$$
    and thus
    $$ \vol( B_\md(x_i, \textstyle \frac{\eps}{2}) \cap (x_i + R \cdot \edball ) ) \geq \frac{3}{4}\vol(x_i + R \cdot \edball ).$$
    Since the set $S_i = \left \{ x \in (x_i + R \cdot \edball ) \colon \: \langle x-x_i, x_i \rangle > 0 \right \}$ is a half-sphere, we get
    $$ \vol(G_i) \geq \vol( B_\md(x_i, \textstyle \frac{\eps}{2}) \cap (x_i + R \cdot \edball ) )  - \vol(S_i) \geq \frac{3}{4}\vol(x_i + R \cdot \edball ) - \frac{1}{2}\vol(x_i + R \cdot \edball ).$$
    So by translation invariance and homogeneity of the Lebesgue measure it follows
    $$ \vol(G_i) \geq \frac{R^d}{4}\vol(\edball).$$
   This yields the estimate
   $$ \mathcal{N}(\edball, \md, \eps) \leq \frac{ (1+R^2)^{\frac{d}{2}} \vol( \edball )}{\frac{1}{4} R^d \vol( \edball )} \leq 4 e^{\frac{d}{2R^2}} =  4 e^{\frac{d}{2C\min \{ \eps^2 , \eps \}}},$$
   which allows us to bound the log covering numbers.
   $$ \sqrt{\log(\mathcal{N}(\edball, \md, \eps))} \lesssim 1 + \frac{\sqrt{d}}{\min \{ \sqrt{\eps},\eps \} }$$
   This estimate performs better when $\eps$ is large, but in the regime $0 < \eps \leq 1$ we will replace the estimate by 
   $$ \mathcal{N}(\edball, \md, \eps) \lesssim\Big( \frac{3d}{\eps} \Big)^d,$$
   which follows from the inequality $\md(u,v) \lesssim \sqrt{d}\norm{u-v}_2$ and using the covering numbers of the euclidean sphere. Plugging these bounds into the entropy integral then gives
   $$ \int_{0}^\infty \sqrt{\log(\mathcal{N}(\edball, \md, \eps))} \, d \eps \lesssim d^{\frac{3}{4}} ,$$
\end{par}
where we compute the above integration only up to $\varepsilon = \sqrt{d}$ in above, as the radius of $\edball$ in $\md$ is at most $\sqrt{d}$.

\section{Applications of Our Results}
\label{sec:examples}

\begin{par}
In this section, we present a few applications of \Cref{thm:masterthm} to different tensor models.
We discuss strengths and weaknesses of the results obtained from \Cref{thm:masterthm} and compare them with other known approaches. It is important to keep in mind that our bounds follow from a completely general theorem that does not exploit additional model assumptions.  
\end{par}

\subsection{The Independent Entry Model and Tensor PCA}
\label{subsec:ind_ent_tensor_PCA} 

We study the class of tensors with independent entries, where the variances of the entries potentially differ. We discuss the implications of \Cref{thm:masterthm} for these tensors and its application to the setting of censored (or partial information) tensor PCA. The proofs are deferred to the end of this subsection for the sake of exposition. 

For a symmetric Gaussian tensor with nonhomogeneous independent entries, the parameters from \Cref{thm:masterthm} can be expressed in terms of its variance tensor as follows: 

\begin{thm}[Nonhomogeneous independent entry model]\label{thm:indepentryapplication}
    Let $T \in (\R^d)^{\otimes r}$ be a symmetric Gaussian tensor with $\E[T] =0$, such that the entries $T_{i_1, \ldots, i_r}$ are independent for $i_1 \leq i_2 \leq \ldots \leq i_r$. Define the variance tensor $A\in (\R^d)^{\otimes r}$ as $A_{i_1, \ldots,i_r} \coloneqq \E[T_{i_1, \ldots,i_r}^2]$ and let $\mathbf{1} \in \R^d$ denotes the vector with all-$1$ coordinates. Then we have
    $$ d^{\frac1p -\frac12} \E \norminp{T} \cleq{r,p} (\log d) \,  \norminx{A \mathbf{1}}{p/2}^{\frac12}   +  \max_{2 \leq q \leq r} \norminx{A \mathbf{1}^{\otimes q}}{p/2} ^{\frac{1}{2q}} \norminx{A }{p/2} ^{\frac{q-1}{2q}}. $$
\end{thm}

In the statement above, recall that $A \mathbf{1}^{\otimes q}$ means contracting $A$ with $\mathbf{1}^{\otimes q}$ as in \Cref{subsec:notation}. 
In the case of $p=2$, this bound simplifies further, since the $\ell_1$ injective norm is the largest entry of the tensor. We illustrate this simplification when $A$ is the adjacency tensor of a hypergraph $H = ([d],E)$, so $A_{i_1, \ldots, i_r} = 1$ if $\{i_1, \ldots, i_r\} \in E$ and $A_{i_1, \ldots, i_r} = 0$ otherwise. 

\begin{cor}\label{cor:hypergraphindepentry}
    Consider the setting in \Cref{thm:indepentryapplication} and assume that $A$ is the adjacency tensor of a hypergraph $H = ([d],E)$. Let $\Delta_j$ be the maximum number of hyperedges in $E$, such that their joint intersection has cardinality at least $j$, then
    $$  \E \norminx{T}{2} \cleq{r} (\log d) \,  \Delta_{r-1}^{\frac12}  +  \max_{2 \leq q \leq r} \Delta_{r-q} ^{\frac{1}{2q}}. $$
\end{cor}

\medskip
\noindent \textbf{Application to Censored Tensor PCA.}
As an application of \Cref{cor:hypergraphindepentry}, we give a brief description of censored tensor PCA: let $\lambda > 0$ and $v \in \{-1,1\}^d$ be a signal vector and $T$ as in \Cref{cor:hypergraphindepentry}. We observe the noisy tensor
$$ Y(\lambda) = (A \odot  \lambda v^{\otimes r}) + T.$$
The tensor $Y(\lambda)$ only has nonzero entries, where $A$ is nonzero, so $A$ determines how many measurements are available. We are interested in the question whether one can statistically detect the presence of the signal given the measurement tensor $Y(\lambda)$, so whether we can distinguish between $Y(\lambda)$ and $Y(0)$. One can take the injective $\ell_2$ norm as a potential statistic to differentiate between these two distributions.

\begin{thm}[Censored tensor PCA]\label{thm:lambdastatpca}
    Consider the setting in \Cref{cor:hypergraphindepentry}. There exist constants $C_r',C_r > 0$, such that detection is possible when 
    $$\lambda \geq \lambda_A \coloneqq C_r \norminx{ A  }{2}^{-1}  \Big[ (\log d) \,  \Delta_{r-1}^{\frac12}  +  \max_{2 \leq q \leq r} \Delta_{r-q} ^{\frac{1}{2q}} \Big] .$$
   In particular, with high probability as $d \to \infty$, the following inequalities hold: 
    \begin{equation} \label{eq:tensor_pca_bounds}
    \begin{split}
        \norminx{(A \odot \lambda v^{\otimes r}) + T }{2} > C_r'  \norminx{A}{2} \lambda_A ,
        \\
        \norminx{ T }{2} < C_r'  \norminx{A}{2} \lambda_A .
    \end{split} 
    \end{equation}
\end{thm}

The question of signal detection with partial information has been studied before in the matrix case (e.g., \cite{ABBS14,C15,BCSvH24}). 
For higher order tensors, to the best of our knowledge, tensor PCA has only been studied in the full-information case, 
where more sophisticated results have been established that also capture the correct dependency of the critical threshold with respect to $r$. While we did not optimize for the dependency on $r$, \Cref{thm:lambdastatpca} does recover the information-theoretically optimal threshold $\lambda_{stat}$ for fixed $r$ (which is $\asymp_r d^{\frac{1-r}{2}}$~\cite{RM14,LML+17}) up to a logarithmic factor in $d$, in the full-information case. Indeed, suppose $A = \mathbf{1}^{\otimes r}$, then $\norminx{A}{2} = d^{\frac{r}{2}}$ and $\Delta_{r-q} \asymp_r d^{q}$, so plugging these values into \Cref{thm:lambdastatpca} yields $\lambda_A \asymp_r (\log d)d^{\frac{1-r}{2}}$. We mention in passing that the computational thresholds for tensor PCA \cite{HSS15,WEAM19}, and the statistical thresholds when $v$ is drawn from priors \cite{PWB20} have also been studied.  

\begin{par}
    Before we proceed with the proofs, we emphasize that the bounds in this section all follow from \Cref{thm:masterthm}, which does not exploit any independent entry assumption. We suspect that by leveraging this assumption, it is possible to use specialized techniques to refine our bounds. We discuss more in \Cref{subsection:refineindepentry}. 
\end{par}

\begin{proof}[Proof of \Cref{thm:indepentryapplication}]
    To apply \Cref{thm:masterthm}, it suffices to bound the tensor parameters as $\sigma_{q, \I_p}^2 = \norminp{H} \lesssim_r \norminx{A\mathbf{1}^{\otimes q}}{p/2}$, where we denote $H := \E[T \star_q T]$. 
    To prove this, note that for any vectors $x_1, \ldots,x_{r-q},y_1, \ldots, y_{r-q} \in \pdball$, we have
    \begin{align*}
    J & := H[x_1, \ldots,x_{r-q},y_1, \ldots, y_{r-q}]  \\
    & = \sum_{ \substack{i_1, \ldots, i_{r-q} = 1 \\ j_1, \ldots, j_{r-q} = 1 } }^d H_{i_1, \ldots, i_{r-q},j_1, \ldots, j_{r-q} }(x_1)_{i_1} \cdots (x_{r-q})_{i_{r-q}} (y_1)_{j_1} \cdots (y_{r-q})_{j_{r-q}} .
    \end{align*}
    Since $T$ is symmetric with independent entries, $H_{i_1, \ldots, i_{r-q},j_1, \ldots, j_{r-q} }$ is given as follows: 
\begin{itemize}
    \item If there is a permutation $\tau \in S_{r-q}$, such that $(i_1, \ldots, i_{r-q}) = (j_{\tau(1)}, \ldots, j_{\tau(r-q)})$, then 
    \[
    H_{i_1, \ldots, i_{r-q},j_1, \ldots, j_{r-q}} = \sum_{k_1,\ldots, k_q=1}^d A_{i_1, \ldots, i_{r-q},k_1, \ldots, k_q} = (A \mathbf{1}^{\otimes q})_{i_1, \ldots, i_{r-q}} .
    \]
    \item If there is no such permutation as above, then $H_{i_1, \ldots, i_{r-q},j_1, \ldots, j_{r-q}} = 0$. 
\end{itemize}
Thus for fixed $(i_1, \ldots, i_{r-q})$, the summand in $J$ is only non-zero for indices $(j_1, \ldots, j_{r-q})$, such that $(i_1, \ldots, i_{r-q}) = (j_{\tau(1)}, \ldots, j_{\tau(r-q)})$ holds for some permutation $\tau$. Therefore, 
\begin{align*}
J 
&\leq \sum_{ \substack{i_1, \ldots, i_{r-q} =1 } }^d (A \mathbf{1}^{\otimes q})_{i_1, \ldots, i_{r-q}} (x_1)_{i_1} \cdots (x_{r-q})_{i_{r-q}}   \sum_{\tau \in S_{r-q}}(y_{\tau(1)})_{i_1} \cdots (y_{\tau(r-q)})_{i_{r-q}} \\
& = \sum_{\tau \in S_{r-q}} \sum_{ \substack{i_1, \ldots, i_{r-q} =1 } }^d  (A \mathbf{1}^{\otimes q})_{i_1, \ldots, i_{r-q}} \big(x_1 \odot y_{\tau(1)}\big)_{i_1} \cdots \big(x_{r-q} \odot y_{\tau(r-q)}\big)_{i_{r-q}}  ,
\end{align*}
where $\odot$ denotes entry-wise multiplication. When upper bounding $J$, we assumed without loss of generality, that all vectors have nonnegative entries. Since the vectors $x_1, \ldots, x_{r-q}$, $y_1, \ldots, y_{r-q} \in \pdball$, it follows that $x_i \odot y_{\tau(i)} \in \mathbb{B}_{p/2}^d$ by the Cauchy-Schwarz inequality. Consequently, for any fixed permutation $\tau \in S_{r-q}$, we have the bound
\[
\sum_{ \substack{i_1, \ldots, i_{r-q} =1 } }^d  (A \mathbf{1}^{\otimes q})_{i_1, \ldots, i_{r-q}} \big(x_1 \odot y_{\tau(1)}\big)_{i_1} \cdots \big(x_{r-q} \odot y_{\tau(r-q)}\big)_{i_{r-q}} \leq \norminx{A \mathbf{1}^{\otimes q}}{p/2} .
\]
This shows that $J \leq (r-q)! \cdot \norminx{A \mathbf{1}^{\otimes q}}{p/2}$ for any choice of $x_1, \ldots,x_{r-q},y_1, \ldots, y_{r-q} \in \pdball$. Thus $\norminp{H} \lesssim_r \norminx{A\mathbf{1}^{\otimes q}}{p/2}$, from which the theorem follows. 
\end{proof}

\begin{proof}[Proof of \Cref{cor:hypergraphindepentry}]
    We start by showing that the $\ell_1$ injective norm of a tensor $U$ is equal to the maximum absolute value of its entry. The lower bound $\max_{i_1, \ldots i_r \in [d]} \abs{U_{i_1, \ldots i_r}} \leq \norminx{U}{1}$ is clear. For the other direction, let $ \norm{x_1}_1 , \ldots, \norm{x_r}_1 \leq 1$, 
    \[\Big| U[x_1, \ldots ,x_r] \Big| \leq \max_{i_1, \ldots i_r \in [d]} \abs{U_{i_1, \ldots i_r}} \sum_{i_1, \ldots, i_r =1}^d  \abs{(x_1)_{i_1} \cdots  (x_r)_{i_r} } \leq \max_{i_1, \ldots i_r \in [d]} \abs{U_{i_1, \ldots i_r}}.
    \]
    Next, we show that the biggest entry of $A \mathbf{1}^{\otimes q}$ is $ \asymp_r \Delta_{r-q}$. Let $\{i_1, \ldots, i_{r-q} \}$ be $r-q$ vertices and $h_1, \ldots, h_l \in E$ be all hyperedges containing $\{i_1, \ldots, i_{r-q} \}$. 
    We have 
    \[
    (A \mathbf{1}^{\otimes q})_{i_1, \cdots, i_r} = \sum_{k_1, \ldots, k_q = 1}^d A_{i_1, \ldots, i_{r-q},k_1, \ldots, k_q} . 
    \]
    Note that $A_{i_1, \ldots, i_{r-q},k_1, \ldots, k_q}$ is $1$ iff $\{i_1, \ldots, i_{r-q},k_1, \ldots, k_q\} = h_j$ for some $j \in [\ell]$. For each hyperedge $h_j$, there are at most $r^q$ choices of $k_1, \cdots, k_q$ so that $\{i_1, \ldots, i_{r-q},k_1, \ldots, k_q\} = h_j$. Consequently, we have 
    \[
    (A \mathbf{1}^{\otimes q})_{i_1, \cdots, i_r} \ceq{r} \ell \leq \Delta_{r-q} . 
    \]
    Maximizing over $i_1, \cdots, i_r$ gives $\|A \mathbf{1}^{\otimes q}\|_1 \asymp_r \Delta_{r-q}$, which implies the corollary. 
\end{proof}

\begin{proof}[Proof of \Cref{thm:lambdastatpca}]
    We apply \Cref{cor:hypergraphindepentry} to get 
    $$\E \norminx{T}{2} \cleq{r} (\log d) \,  \Delta_{r-1}^{\frac12}  +  \max_{2 \leq q \leq r} \Delta_{r-q} ^{\frac{1}{2q}}.$$
    We can use tail bounds for Gaussian processes (Theorem 2.7.13 in~\cite{Tal21}) to show that this bound also holds with high probability. Since $ \norminx{T}{2} = \sup_{x_1, \ldots, x_r \in \edball} \langle T, x_1 \otimes \cdots \otimes x_r \rangle$ is the supremum of a Gaussian process, its diameter can be bounded as 
    $$ \E \langle T , x_1 \otimes \cdots \otimes x_r - x_1' \otimes \cdots \otimes x_r' \rangle^2  \cleq{r} \normf{x_1 \otimes \cdots \otimes x_r - x_1' \otimes \cdots \otimes x_r'}^2 \cleq{r} 1 .$$
    Thus applying Theorem 2.7.13 in~\cite{Tal21}, we have that with probability $1-1/\poly(d)$,
    $$\norminx{T}{2} \cleq{r} (\log d) \,  \Delta_{r-1}^{\frac12}  +  \max_{2 \leq q \leq r} \Delta_{r-q} ^{\frac{1}{2q}} + \sqrt{\log(d)} < K_r \left[ (\log d) \,  \Delta_{r-1}^{\frac12}  +  \max_{2 \leq q \leq r} \Delta_{r-q} ^{\frac{1}{2q}} \right].$$
    This proves the second bound in \eqref{eq:tensor_pca_bounds}. 
    To show that the first bound in \eqref{eq:tensor_pca_bounds} also holds with high probability, we simply use the triangle inequality.
    $$\norminx{(A \odot \lambda v^{\otimes r}) + T }{2} > \norminx{A \odot \lambda v^{\otimes r}}{2} -  K_r \left[ (\log d) \,  \Delta_{r-1}^{\frac12}  +  \max_{2 \leq q \leq r} \Delta_{r-q} ^{\frac{1}{2q}} \right] .$$
    Without loss of generality, we may assume $v = \mathbf{1}$, since $\norminx{A \odot \lambda v^{\otimes r}}{2}$ remains invariant under the choice of $v \in \{-1,1\}^d$, so we have  $\norminx{A \odot \lambda 1^{\otimes r}}{2} = \lambda \norminx{A }{2}$. When $\lambda \geq \lambda_A$ for large enough constant $C_r$, it then follows that 
    \[
    \norminx{A \odot \lambda v^{\otimes r}}{2} >  2 K_r \left[ (\log d) \,  \Delta_{r-1}^{\frac12}  +  \max_{2 \leq q \leq r} \Delta_{r-q} ^{\frac{1}{2q}} \right] ,
    \]
    from which the second statement in \eqref{eq:tensor_pca_bounds} follows. This proves the theorem.  
\end{proof}

\subsection{Matching Matrices}
\label{subsec:matching_mat}

\begin{par}

In this subsection, we give an example of structured random matrices where the bound implied by \Cref{thm:masterthm} is superior to those from other approaches in certain regimes. 

   We consider the structured random matrix model where every matrix in \eqref{eq:Tasgaussianseries} is a matching matrix. A matching $M$ on $[d]$ is a collection of disjoint unordered pairs (also called edges) in $[d]$. 
   The adjacency matrix $A$ of matching $M$ is defined as $A_{i,j} = 1$ if $\{i, j \} \in M$ and $A_{i,j} = 0$ otherwise. \Cref{thm:masterthm} implies the following upper bound for such matrices when $p \geq 4$ (we defer its proof to the end of the subsection).

\begin{thm}\label{thm:matchingmatsourbound}
    Let $M_1, \ldots, M_n$ be matchings on $[d]$, the respective adjacency matrices be $A_1, \ldots, A_n \in \R^{d \times d}$, and $E \coloneqq M_1 \sqcup \cdots \sqcup M_n$. Define $\mu \in \R^n$ by $\mu_i \coloneqq \abs{M_i}$ and $\Delta \in \R^d$ such that $\Delta_i$ is the degree of vertex $i$ in the multigraph $([d],E)$ (including multiplicities). For $p \geq 4$ and $g_1, \ldots, g_n$ being i.i.d. standard Gaussians, we have
    \begin{equation}\label{eq:matchingmatsourineq}
        d^{\frac1p -\frac12} \E \Big\|\sum_{k=1}^n g_k A_k \Big\|_{\I_p} \cleq{p} (\log d)\norm{\mu}_1^{\frac14} \norm{\Delta }_\infty^{\frac 1p} \norm{\mu}_{\frac{2p-4}{p-4}}^{\frac{1}{2} - \frac{1}{p}}.
    \end{equation}
\end{thm}
Since we only consider matrices in this example, we could also use the noncommutative Khintchine inequality \eqref{eq:NCK} together with Hölder's inequality as in \Cref{rem:nckimprovedtype2}. 
This second natural approach yields the upper bound
\begin{equation}\label{eq:ncklpmatchingestim}
    d^{\frac1p -\frac12} \E \Big\|\sum_{k=1}^n g_k A_k \Big\|_{\I_p} \lesssim  \sqrt{\log d} \cdot d^{\frac12 -\frac1p}\norm{\Delta}_\infty^{\frac 12}.
\end{equation}
Before we prove these inequalities, we discuss when the bound in \Cref{thm:matchingmatsourbound} is better than \eqref{eq:ncklpmatchingestim}. Ignoring logarithmic factors, this is the case when
\begin{equation}\label{eq:secondtermmatchingcondition}
    \norm{\mu}_1^{-\frac{p-4}{2p-4}} \norm{\mu}_{\frac{2p-4}{p-4}} \ll d \norm{\Delta}_\infty \norm{\mu}_1^{-1}.
\end{equation}
The quantity on the right-hand side is the ratio of the maximum degree and the average degree in the multigraph $([d],E)$. 
Since the left-hand side of \eqref{eq:secondtermmatchingcondition} is upper bounded by $\norm{\mu}_\infty^{\frac{p}{2p-4}}$ using Hölder's inequality, the gap in \eqref{eq:ncklpmatchingestim} is sufficiently large whenever \[
\norm{\mu}_\infty^{\frac{p}{2p-4}} \ll d \norm{\Delta}_\infty \norm{\mu}_1^{-1} .
\]
Thus, if $([d],E)$ is irregular and the matchings $M_i$ are sufficiently sparse, then~\eqref{eq:matchingmatsourineq} improves upon~\eqref{eq:ncklpmatchingestim}. Intuitively it makes sense that there needs to be enough asymmetry in the multigraph $([d],E)$ to significantly improve upon the estimate~\eqref{eq:ncklpmatchingestim}: if the injective $\ell_2$ is maximized by vectors, whose coordinates all have similar magnitude, the application of Hölder's inequality underlying \eqref{eq:ncklpmatchingestim} is essentially lossless. 

An ideal estimate for the injective norm should be as dimension-free as possible.  
While we have successfully reduced the dimensional factor in~\eqref{eq:matchingmatsourineq}, it remains unclear to us what the appropriate variance parameters should be to further reduce it.
\end{par}

\begin{rem}
    By the discussion above, \Cref{thm:matchingmatsourbound} may not be optimal. However, not all of the loss is due to \Cref{thm:masterthm}. Some of the parameter estimations are purposefully loose to express the upper bound in terms of more explicit parameters of the matchings.
\end{rem}

\begin{proof}[Proof of \Cref{thm:matchingmatsourbound}]
    We estimate the variance parameters for the model. In the case of $\pvar{2}$, we have an exact expression.
    \begin{equation}\label{eq:matchvar2}
        \pvar{2}^2= \sum_{k=1}^n \normf{A_i}^2 = 2 \sum_{k=1}^n \abs{M_k} = 2 \norm{\mu}_1
    \end{equation}
    The parameter $\pvar{1}$ can be computed using the following observation: since the $A_k$ are adjacency matrices of matchings, their squares are diagonal matrices with $(A_k^2)_{j,j} = 1$ if $M_k$ has an edge containing $j$ and $(A_k^2)_{j,j} = 0$ otherwise. Thus we have 
    $$ \sum_{k=1}^n (A_k^2)_{j,j} = \Delta_j . $$
    Note that $\norm{x \odot y}_{\frac p2} \leq 1$ for any $x,y \in \pdball$ by the Cauchy-Schwarz inequality. We then get 
    \begin{equation}\label{eq:matchvar1}
    \pvar{1}^2 = \sup_{x,y \in \pdball} \sum_{j=1}^d \Delta_j x_j y_j \leq \sup_{z \in \mathbb{B}_{p/2}^d} \sum_{j=1}^d \Delta_j z_j = \norm{\Delta}_{\frac{p}{p-2}}.
    \end{equation}

    It remains to estimate the parameter $\pvar{0}$. Contrary to the other variance parameters, we have no precise expression for $\pvar{0}$ in terms of simple properties of the matchings. The following estimate can be useful nevertheless.
    $$ \pvar{0}^2 = \Big\|\sum_{k=1}^n A_k \star_0 A_k\Big\|_{\I_p} = \sup_{x,y,x',y' \in \pdball}  \sum_{k=1}^n \langle A_k, x \otimes y \rangle \langle A_k, x' \otimes y' \rangle $$
    Using the Cauchy-Schwarz inequality on the sum over $k$, we can get rid of the maximum over $x',y' \in \pdball$.
    $$ \pvar{0}^2 \leq \sup_{x,y \in \pdball}  \sum_{k=1}^n \langle A_k, x \otimes y \rangle^2 = \sup_{x,y \in \pdball} \sum_{k=1}^n \bigg( \sum_{i,j=1}^d (A_k)_{i,j} x_i y_j\bigg)^2 $$
    We will apply Hölder's inequality on the sum over $i,j$ with Hölder conjugates $\textstyle \frac{2}{p} + \frac{p-2}{p} = 1$. We separate the terms as $(A_k)_{i,j} x_i y_j = (A_k)_{i,j} \cdot (A_k)_{i,j} x_i y_j$, which holds since the matrix $A_k$ only has entries in $\{0, 1 \}$.
    $$ \pvar{0}^2 \leq \sup_{x,y \in \pdball} \sum_{k=1}^n \bigg( \sum_{i,j=1}^d \abs{(A_k)_{i,j}}^{\frac{p}{p-2}} \bigg)^{ \frac{p-2}{p} \cdot 2}\bigg( \sum_{i,j=1}^d \abs{(A_k)_{i,j} x_i y_j}^{\frac p2}\bigg)^{ \frac 2p \cdot 2 } $$
    We simplify this expression by noting that the sum of all entries of $A_k$ is twice the number of edges in $M_k$.
    $$ \pvar{0}^2 \leq \sup_{x,y \in \pdball} \sum_{k=1}^n \left( 2 \mu_k \right)^{ \frac{2p-4}{p} }\bigg( \sum_{i,j=1}^d \abs{(A_k)_{i,j} x_i y_j}^{\frac p2}\bigg)^{ \frac 4p  } $$
    Applying Hölder's inequality on the sum over $k$ with Hölder conjugates $\textstyle \frac{4}{p} + \frac{p-4}{p} = 1$ and separating the $\mu_k$ factors yields the following:
    $$ \pvar{0}^2 \leq \sup_{x,y \in \pdball} \bigg(\sum_{k=1}^n (2\mu_k)^{\frac{2p-4}{p-4}} \bigg)^{\frac{p-4}{p}} \bigg( \sum_{k=1}^n \sum_{i,j=1}^d \abs{(A_k)_{i,j} x_i y_j}^{\frac p2}\bigg)^{\frac 4p} $$
    The first factor is now independent of $x,y$. In the second factor we use the AM-GM inequality on $\abs{x_i y_j}^{\frac p2}$ to bound $\abs{(A_k)_{i,j} x_i y_j}^{\frac p2} \leq \frac{1}{2}(A_k)_{i,j}(\abs{x_i}^p + \abs{y_j}^p )$. (The fact that $A_k$ has entries in $\{0, 1 \}$ was also used here.)

    $$ \pvar{0}^2 \lesssim \norm{\mu}_{\frac{2p-4}{p-4}}^\frac{2p-4}{p} \cdot \sup_{x,y \in \pdball} \bigg( \frac12 \sum_{k=1}^n\sum_{i,j=1}^d (A_k)_{i,j}(\abs{x_i}^{p} + \abs{ y_j}^{p})\bigg)^{\frac 4p} $$
    Note that we have a sum of two identical suprema by separating the terms with $x_i$ coefficients from the ones with $y_j$ and using that $A_k$ is symmetric. Moreover, for a fixed $i$ we have $\sum_{k \in [n], j\in [d]}^d (A_k)_{i,j} = \Delta_i$ and hence
    \begin{equation}\label{eq:matchvar0}
        \pvar{0}^2 \lesssim \norm{\mu}_{\frac{2p-4}{p-4}}^\frac{2p-4}{p} \cdot \bigg( \sup_{x \in \pdball} \sum_{i=1}^d \Delta_i\abs{x_i}^{p} \bigg)^{\frac 4p} = \norm{\mu}_{\frac{2p-4}{p-4}}^\frac{2p-4}{p}  \norm{\Delta}_\infty^{\frac 4p}.
    \end{equation}
     Plugging~\eqref{eq:matchvar0},\eqref{eq:matchvar1},\eqref{eq:matchvar2} into \Cref{thm:masterthm} yields 
    $$ d^{\frac1p -\frac12} \E \Big\|\sum_{k=1}^n g_k A_k \Big\|_{\I_p} \cleq{p} (\log d)\norm{\Delta}_{\frac{p}{p-2}}^{\frac12}+ \norm{\mu}_1^{\frac14} \norm{\Delta }_\infty^{\frac 1p} \norm{\mu}_{\frac{2p-4}{p-4}}^{\frac{2p-4}{4p}}.$$
    To finish the proof, it remains to show that the first term is dominated by the second term. Using Hölder's inequality and the fact that $2\norm{\mu}_1 = \norm{\Delta}_1$ holds, we get
    $$ \norm{\Delta}_{\frac{p}{p-2}}^{\frac12} = (\norm{\Delta}_{\frac{p}{p-2}}^{\frac{p}{p-2}})^{\frac{p-2}{2p}} \leq \norm{\Delta}_\infty^{\frac1p} (\norm{\Delta}_1)^{\frac{p-2}{2p}} \asymp \norm{\Delta}_\infty^{\frac1p} \norm{\mu}_1^{\frac{p-2}{2p}} = \norm{\Delta}_\infty^{\frac1p} \norm{\mu}_1^{\frac14}  \norm{\mu}_1^{\frac{p-4}{4p}}.  $$
    Since the vector $\mu$ only has integer coordinates, taking its coordinates to a power greater than $1$ increases their magnitude and thus
    $$ \norm{\Delta}_{\frac{p}{p-2}}^{\frac12} \lesssim \norm{\Delta}_\infty^{\frac1p} \norm{\mu}_1^{\frac14}  \norm{\mu}_1^{\frac{p-4}{4p}} \leq  \norm{\Delta}_\infty^{\frac1p} \norm{\mu}_1^{\frac14}  \left( \norm{\mu}_{\frac{2p-4}{p-4}}^{\frac{2p-4}{p-4}} \right)^{\frac{p-4}{4p}} = \norm{\mu}_1^{\frac14} \norm{\Delta }_\infty^{\frac 1p} \norm{\mu}_{\frac{2p-4}{p-4}}^{\frac{2p-4}{4p}}. $$
\end{proof}

\begin{proof}[Proof of~\eqref{eq:ncklpmatchingestim}]
    By Hölder's inequality, $\pdball \subseteq d^{\frac12 - \frac1p} \edball$, so replacing the injective $\ell_p$ norm by the injective $\ell_2$ norm and then applying NCK in \eqref{eq:NCK} gives
    $$  \E \Big\|\sum_{k=1}^n g_k A_k\Big\|_{\I_p} \leq \left( d^{\frac12 -\frac1p} \right)^{2}\Big\|\sum_{k=1}^n g_k A_k\Big\|_{\I_2} \lesssim  \sqrt{\log d} \cdot d^{1 -\frac2p} \Big\|\sum_{k=1}^n A_k^2\Big\|_{\I_2}^{\frac12}. $$
    In the proof of \Cref{thm:matchingmatsourbound}, we discussed that
    $ \sum_{k=1}^n A_k^2 $
    is diagonal with $\Delta_i$ being the $i$-th diagonal element. Thus, its spectral norm corresponds to its largest entry, so
    $$  \E \Big\|\sum_{k=1}^n g_k A_k\Big\|_{\I_2} \lesssim  \sqrt{\log d} \cdot d^{1 -\frac2p} \norm{\Delta}_\infty^{\frac 12}. $$
\end{proof}

\section{Preliminaries}\label{section:preliminaries}

\subsection{Variance Parameters and Tensor Norm Inequalities}

\begin{par}
    For a symmetric Gaussian tensor of the form $T =\sum_{k=1}^n \rv{g}_k T_k$, we introduced in \Cref{defn:intovariances} a number of variance parameters (whose dependency on the tensors $T_1, \ldots, T_n$ is omitted when it is clear from the context). The next result presents equivalent formulations of these parameters, whose proof closely follows \cite[Lemma A.4]{Kevinsthesis}.
\end{par}

\begin{lemma}\label{lemma:variances}
    Given $T_1, \ldots, T_n \in (\R^d)^{\otimes r}$ symmetric tensors, $ T = \sum_{k=1}^n \rv{g}_k T_k $ a Gaussian tensor as in~\eqref{eq:Tasgaussianseries}, and $0 \leq q \leq r$ we have
    $$ \pvar{q}^2 = \Big\|\sum_{k=1}^n T_k \star_q T_k\Big\|_{\I_p} = \sup_{u_l \in \pdball, \: 1 \leq l \leq r-q} \sum_{k=1}^n \normf{T_k [u_1, \ldots, u_{r-q}, \cdot, \ldots, \cdot]}^2. $$
\end{lemma}

\begin{proof}
    The first inequality follows from the bilinearity of the $\star_q$ product.
    $$ \E[T \star_q T] = \E\bigg[ \Big(\sum_{k=1}^n \rv{g}_k T_k\Big) \star_q  \Big(\sum_{k=1}^n \rv{g}_k T_k\Big) \bigg] = \sum_{k,k'=1}^n \E[\rv{g}_k \rv{g}_{k'}] T_k  \star_q T_{k'} = \sum_{k=1}^n T_k \star_q T_k .$$
    For the second equality, we first show that the middle term is upper bounded by the right term. Let $u_1, \ldots, u_{r-q}, v_{q+1}, \ldots, v_{r} \in \pdball$, we have 
    \begin{equation}\label{equation:contractedmultinorm}
        \begin{split}
            &\Big( \sum_{k=1}^n T_k \star_q T_k \Big)[u_1, \ldots, u_{r-q}, v_{q+1}, \ldots, v_r] 
            \\
            &= \sum_{k=1}^n \sum_{m_1,\ldots, m_q=1}^d T_k [u_1, \ldots, u_{r-q}, e_{m_1}, \ldots, e_{m_q} ] T_k[e_{m_1}, \ldots, e_{m_q}, v_{q+1}, \ldots, v_r] \quad \mbox{(Defin. of $\star_q$)}
            \\
            &= \sum_{k=1}^n \big\langle T_k [u_1, \ldots, u_{r-q}, \cdot , \ldots, \cdot] , T_k[ \cdot ,\ldots, \cdot, v_{q+1}, \ldots, v_r] \big \rangle \quad \mbox{(Defin. of dot product)}
            \\
            &\leq \sum_{k=1}^n \normf{T_k [u_1, \ldots, u_{r-q}, \cdot , \ldots, \cdot]} \cdot \normf{T_k[ \cdot ,\ldots, \cdot, v_{q+1}, \ldots, v_r]} \quad \mbox{(Cauchy-Schwarz on $\langle \cdot, \cdot \rangle$)}
            \\
            &\leq \sqrt{ \sum_{k=1}^n \normf{T_k [u_1, \ldots, u_{r-q}, \cdot , \ldots, \cdot]}^2 \sum_{k'=1}^n \normf{T_{k'}[ \cdot ,\ldots, \cdot, v_{q+1}, \ldots, v_r]}^2} \quad \mbox{(Cauchy-Schwarz)}
        \end{split}
    \end{equation}
    The last step used the Cauchy-Schwarz inequality on the sum over the index $k$. Using symmetry of $T_k$ and taking the supremum over the $u_l$ and $v_l$ being in $\pdball$ yields:
    \begin{equation}\label{eq:variancesup}
    \Big\|\sum_{k=1}^n T_k \star_q T_k \Big\|_{\I_p} \leq \sup_{u_l \in \pdball, \: 1 \leq l \leq r-q } \sum_{k=1}^n \normf{T_k [u_1, \ldots, u_{r-q}, \cdot, \ldots, \cdot]}^2
    \end{equation}
    The $v_l$ and the square-root disappeared in this expression, since we are multiplying two identical suprema. To show that~\eqref{eq:variancesup} is an equality, we pick our $v_l$ such that $u_{l}= v_{l+q}$ holds for $1 \leq l \leq r-q$. Then all inequalities in~\eqref{equation:contractedmultinorm} become equalities, since the vectors for which we used the Cauchy-Schwarz inequality would be identical. Adding this restriction will end up with the same supremum as the one in~\eqref{eq:variancesup} due to the symmetry of the $T_k$, which proves equality of the two expressions.
\end{proof}

\begin{par}
    It is not immediately clear how these parameters relate to the injective norms of the tensors $T_1, \ldots, T_n$, and we will investigate this in the remainder of this subsection. In the case $p=2$ and $r=2$ these parameters have already appeared in the literature, where $\pvar{0}$ is often referred to as the ``weak variance"~\cite{Tro15} while $\pvar{1}$ is the variance parameter appearing in the noncommutative Khintchine inequality in~\eqref{eq:NCK}. 
    
    In general, relating these variance parameters and the sum of squared injective norms requires comparison bounds between the Frobenius norm and injective norm of a tensor. This turns out to be a specific case of a more general problem about boundedness of multilinear forms in $\ell_p$-spaces, which has been well studied (e.g., \cite{Per81,DS16}), with the oldest source dating back to Hardy and Littlewood, who focused on bilinear forms in \cite{HL34}. We use the following inequality from \cite[Theorem B]{Per81}, which has been reproved and generalized in \cite[Proposition 4.1 and 4.4]{DS16}. 
\end{par}

\begin{thm}[Theorem B in \cite{Per81}]\label{thm:boundedmultilinear}
    Let $T \in (\R^d)^{\otimes r}$ and let $p_1, \ldots, p_r \geq 1$ be reals such that $s \coloneqq \frac{1}{p_1} + \ldots + \frac{1}{p_r} \leq \frac{1}{2}$ holds and define $\tilde{p} = \frac{2r}{r+1 -2s}$. We have
    $$  \bigg( \sum_{m_1, \ldots, m_r =1}^d \abs{(T)_{(m_1, \ldots , m_r)}}^{\tilde{p}} \bigg)^{1/\tilde{p}} \cleq{r} \sup_{ \norm{u_i}_{p_i} = 1 } T[u_1, \ldots, u_r].$$
\end{thm}

\begin{cor}\label{cor:normtransition}
    Let $T$ be a $d$-dimensional order $r$ tensor and $2 \leq p < \infty$. We have
    $$ \normf{T} \cleq{r} d^{\max \{r/p-1/2,0 \}}\norminp{T}.$$
    
\end{cor}

\begin{proof}
    We first consider the case $p = 2r$. We use \Cref{thm:boundedmultilinear} where we define $p_1 = \cdots = p_r = 2r$, so in this case $\tilde{p} =2$ holds and we get
    $$ \normf{T} = \bigg( \sum_{m_1, \ldots, m_r =1}^d \abs{(T)_{(m_1, \ldots , m_r)}}^{2} \bigg)^{1/2} \cleq{r} \sup_{ \norm{u_i}_{2r} = 1 } T[u_1, \ldots, u_r] = \norminx{T}{2r}.$$
    The result for $p \geq 2r$ immediately follows since $\norminx{T}{2r} \leq \norminp{T}$. If $p \leq 2r$, we use Hölder's inequality on the vectors $u_i$ to prove
    $$ \norminx{T}{2r}  =  \sup_{ \norm{u_i}_{2r} = 1 } T[u_1, \ldots, u_r]  \leq  \sup_{ \norm{u_i}_{p} \leq  d^{\frac{1}{p}- \frac{1}{2r} } } T[u_1, \ldots, u_r] = d^{\frac{r}{p}- \frac{r}{2r} } \norminp{T}.$$
\end{proof}

\begin{rem}
The upper bound in \Cref{cor:normtransition} is tight up to a constant factor in general in the regime of $p \geq 2$. 
When $2 \leq p \leq 2r$, a tensor with i.i.d. standard Gaussian entries  proves its tightness (see \Cref{lemma:optimalitynormtransition}); 
for $p \geq 2r$, a tensor with a single nonzero entry shows tightness of the bound.

    While the bound in \Cref{cor:normtransition} also extends to the case of $p < 2$, it would no longer be optimal up to constants. For example, for $p=1$ the injective norm is simply the maximal absolute entry of a tensor, so $\normf{T} \leq d^{r/2}\norminx{T}{1}$, which is much smaller than $d^{r- \frac{1}{2}}$. 
\end{rem}

\begin{par}
    This result allows us to relate the variance parameters to the sum of squared norms.
\end{par}

\begin{defn}\label{defn:type2var}
    Given tensors $T_1, \ldots, T_n \in (\R^d)^{\otimes r}$ we define the type-$2$ $\mathcal{I}_p$-variance as
    $$ \typetwopvar^2 \coloneqq \sum_{k=1}^n \norminp{T_k}^2 .$$
\end{defn}

\begin{cor}\label{cor:type2var}
    For any $0 \leq q \leq r$ and tensors $T_1, \ldots, T_n \in (\R^d)^{\otimes r}$ we have
    $$ \pvar{q} \cleq{r} d^{\max \{q/p-1/2,0 \}} \typetwopvar.$$
\end{cor}

\begin{proof}
    We take the supremum in \Cref{lemma:variances} inside the sum and then use \Cref{cor:normtransition} to bound $ \normf{T_k [u_1, \ldots, u_{r-q}, \cdot, \ldots, \cdot]} \lesssim_r d^{\max \{q/p-1/2,0 \}}\norminp{T}$, since $T_k [u_1, \ldots, u_{r-q}, \cdot, \ldots, \cdot]$ is an order $q$ tensor.
    $$ \pvar{q}^2 \leq  \sum_{k=1}^n \sup_{u_l \in \pdball, \: 1 \leq l \leq r-q, } \normf{T_k [u_1, \ldots, u_{r-q}, \cdot, \ldots, \cdot]}^2 \cleq{r} \sum_{k=1}^n d^{\max \{\frac{2q}{p}-1,0 \}} \norminp{T_k}^2 $$
\end{proof}

\begin{par}
    In the context of random matrices this parameter appeared first for a cruder version of the noncommutative Khintchine inequality in \cite{Tom74}, where Schatten norms were considered. For the spectral norm of matrices it appeared in \cite{AW02}, where the authors proved a cruder version of a matrix Chernoff inequality using the sum of squared operator norms. While there are cases where the upper bound in \Cref{cor:type2var} is tight, the loss one gets by using the type-$2$ $\mathcal{I}_p$-variance can in general be significant. For the $p=r=2$ (Euclidean matrix) case, this loss has been discussed in more detail in~\cite{Tro15}. 
\end{par}

\subsection{The Gaussian Process Setting}

\begin{par}
    Finding the expected injective norm of a symmetric Gaussian tensor $ T = \sum_{k=1}^n \rv{g}_k T_k $ boils down to bounding the expected supremum of the following Gaussian processes.
    $$ \E \bigg[  \Big\|\sum_{k=1}^n \rv{g}_k T_k\Big\|_{\I}\bigg] = \E \bigg[  \sup_{\norm{x_1}, \ldots, \norm{x_r} \leq 1} \sum_{k=1}^n \rv{g}_k \langle \textstyle  T_k, x_1 \otimes \ldots \otimes x_r \rangle \bigg] . $$
    Since our focus is not on optimizing constant factors that depend on $r$, we may consider the following expected supremum instead (using \Cref{prop:symnormequiv} later):
    \begin{align} \label{eq:gauss_process_prelim}\E \bigg[ \sup_{\norm{u} \leq 1} \Big| \sum_{k=1}^n \rv{g}_k \langle \textstyle  T_k, u \otimes \ldots \otimes u \rangle \Big| \bigg] . 
    \end{align}
    The absolute value inside the supremum can be removed up to a factor of $2$ by a symmetry argument (replacing the $\rv{g}_k$ by $-\rv{g}_k$ will not change the distribution and the supremum is always nonnegative as $0$ is in $\pdball$; see also in the proof of \Cref{prop:trivialsigmastarbound}).

    \begin{defn}
        Let $T \subset \R^n$ be a set. A \emph{Gaussian process} $\{g_u\}_{u \in T}$ is a collection of centered jointly Gaussian random variables. 
        The \emph{natural distance} on $T$ induced by the Gaussian process is defined by
        $$ \md(u,v) \coloneqq \sqrt{\E (g_u-g_v)^2} .$$
    \end{defn}

    The celebrated majorizing measures theorem of Talagrand~\cite{Tal21} says that the expected supremum of a Gaussian process (under mild technical assumptions) is determined by the natural distance structure on the process. The so-called $\gamma_2$-functional characterizes  up to a universal constant the expected supremum. 
    $$ \E \sup_{u \in T} g_u \asymp \gamma_2(T,\md).$$ 
    Unfortunately, bounding the $\gamma_2$-functional is a difficult task in general with many open questions \cite{Tal21}. If one does not aim for the correct polylogarithmic factors (as in this paper), the expected supremum of the process can be bounded by the covering numbers of process with respect to the natural distance via Dudley's entropy integral. 

    \begin{defn}[Covering numbers]
        Let $(T,\md)$ be a semi-metric space (i.e., a metric space without the assumption that $\md(x,y) = 0$ implies $x=y$). For any $\eps > 0$, we define $\mathcal{N}(T, \md, \eps)$ to be the minimum cardinality of a subset $N \subseteq T$, such that for any $t \in T$ there exists a $t' \in N$ with $\md(t,t') \leq \eps$. We refer to such a set $N$ as an $\eps$-\emph{covering}, and $\mathcal{N}(T, \md, \eps)$ the \emph{covering numbers} of the space $(T, \md)$.
    \end{defn}
    
    \begin{thm}[Dudley's entropy integral]\label{thm:entropyintegral}
        Let $\{g_u\}_{u \in T}$ be a Gaussian process and $\md$ be the natural distance on $T$. If $T$ contains a countable subset $T'$ for which $\sup_{u \in T} g_u = \sup_{u \in T'} g_u $ holds almost surely, then
        $$ \E \sup_{u \in T} g_u \lesssim \int_0^\infty \sqrt{\log \mathcal{N} (T,\md, \eps)} \,d\eps. $$
    \end{thm}
    
    Providing covering number estimates for the Gaussian process corresponding to tensor injective norms remains a challenging task, even when one is willing to lose logarithmic factors. In the following \Cref{sec:proof}, we present a volumetric argument to bound the covering numbers, which provides the correct estimate (up to a log factor) for some regimes of $r,p$ and yields some sub-optimal, but nevertheless new, estimates for other values. 
\end{par}

\subsection{Classical Covering Number Bounds}
In this subsection, we discuss classical bounds for the covering numbers above and why the resulting tensor concentration inequality is not satisfactory.

    \begin{defn}
    Given symmetric tensors $T_1, \ldots, T_n \in (\R^d)^{\otimes r}$, we define the \emph{natural distance} $\md$ on $\R^d$ as follows:
    $$ \md(u,v) \coloneqq \sqrt{\sum_{k=1}^n (\langle T_k, u^{\otimes r} \rangle - \langle T_k, v^{\otimes r} \rangle)^2 } .$$
    \end{defn}
    
    The indexing set for the process in \eqref{eq:gauss_process_prelim} in our setting will be $\pdball$. We omit the dependency on the $T_k$ in the distance $\md$, as it is usually clear from context. There is a direct relationship between the distance $\md$ and the $\ell_p$ distance. 

\begin{lemma}\label{lemma:cauchyschwarzdistancebound}
    Let $T_1, \ldots, T_n \in (\R^d)^{\otimes r}$ be symmetric tensors and $u,v \in \pdball$, then
    $$ \md(u,v) \leq r\pvar{0}\norm{u-v}_p .$$
\end{lemma}

\begin{proof}
    We start by writing telescopic sums inside the squares of the natural distance.
    $$ \md(u,v)^2 = \sum_{i=1}^n (\langle T_i, u^{\otimes r} \rangle - \langle T_i, v^{\otimes r} \rangle)^2 = \sum_{i=1}^n \Big(  \sum_{q=1}^r \big( \langle T_i, u^{\otimes r-q+1} \otimes v^{\otimes q-1} \rangle - \langle T_i, u^{\otimes r-q} \otimes v^{\otimes q}  \rangle \big) \Big)^2 . $$
    Using the Cauchy-Schwarz inequality on all $n$ squares yields:
    $$ \md(u,v)^2 \leq \sum_{i=1}^n r \sum_{q=1}^r (\langle T_i, u^{\otimes r-q+1} \otimes v^{\otimes q-1} \rangle - \langle T_i, u^{\otimes r-q} \otimes v^{\otimes q}  \rangle)^2 . $$
    Exploiting the symmetry of the $T_i$, we can rewrite the bound above as follows:
    $$ \sum_{i=1}^n r \sum_{q=1}^r \langle T_i, u^{\otimes r-q} \otimes v^{\otimes q-1} \otimes (u-v)\rangle^2 =  r\norm{u-v}_p^2 \sum_{q=1}^r \sum_{i=1}^n \bigg \langle T_i, u^{\otimes r-q} \otimes v^{\otimes q-1} \hspace{-1mm}\otimes \hspace{-0.5mm}\frac{u-v}{\norm{u-v}_p} \bigg \rangle^2 .$$
    Since $u,v \in \pdball$, the sum over the index $i$ is bounded by $\pvar{0}^2$ (due to \Cref{lemma:variances}). Thus, 
    $$ \md(u,v)^2 \leq r\norm{u-v}_p^2 \sum_{q=1}^r \pvar{0}^2 = r^2 \norm{u-v}_p^2 \pvar{0}^2 .$$
\end{proof}

\begin{par}
    This lemma provides us with the following first covering number estimate. 
\end{par}

\begin{cor}\label{cor:pdlogcovernumberspnorm}
    Given symmetric tensors $T_1, \ldots, T_n \in (\R^d)^{\otimes r}$. For any $0< s <  r\pvar{0} $, 
    $$ \sqrt{\log(\mathcal{N}(\pdball,\md,s))} \leq \sqrt{d\log(3r \pvar{0}/s)} .$$
    For $s \geq r \pvar{0}$, we have 
    $$ \sqrt{\log(\mathcal{N}(\pdball,\md,s))} = 0 . $$
\end{cor}

\begin{proof}
    Let $0 < \eps <1$ and pick a minimal $\eps$-covering of $\pdball$ with respect to the $\ell_p$-distance (setting $\eps=s/(r\pvar{0})$). A classical volumetric argument shows that the covering number is bounded by $(3/\varepsilon)^d$, see~\cite[Corollary 4.1.15]{AsymptoticGeometricAnalysis-Book2015}.
Since an $\eps$-covering of $\pdball$ with respect to $\norm{\cdot}_p$ is an $s$-covering of $\pdball$ with respect to $\md$ by \Cref{lemma:cauchyschwarzdistancebound}, we have
    $$ \sqrt{\log(\mathcal{N}(\pdball,\md,s))} \leq \sqrt{d\log(3r \pvar{0}/s)}.$$
    In the case $ s \geq r \pvar{0}$, the point $0$ already forms an $s$-covering of $\pdball$, which implies $\sqrt{\log(\mathcal{N}(\pdball,\md,s))} = 0$. This proves the corollary. 
\end{proof}

\begin{par}
This covering number bound has been referred to as ``Slepian bound"~\cite{BvH16}, since it can also be shown using a Gaussian comparison inequality.
This bound can be used in the entropy integral, but as we will see below, the resulting tensor concentration inequality is not satisfactory, which also serves as additional motivation for this paper. 
\end{par}

\begin{prop}\label{prop:trivialsigmastarbound}
    Let $T_1, \ldots, T_n \in (\R^d)^{\otimes r}$ be symmetric tensors. For $\rv{g}_1, \ldots, \rv{g}_n$ being i.i.d. standard Gaussians we have
    \[ \E \Big\| \sum_{k=1}^n \rv{g}_k T_k \Big\|_{\I_p} \cleq{r} \sqrt{d} \pvar{0} .\]
\end{prop}

\begin{proof}
    Using \Cref{prop:symnormequiv} and the symmetry of the Gaussians $\rv{g}_k$, it suffices to bound
    \begin{equation*}
    \begin{split}
    \E \Big\| \sum_{k=1}^n \rv{g}_k T_k\Big\|_{\I_p} &\cleq{r} \E \sup_{\norm{u}_p \leq 1} \Abs{ \sum_{k=1}^m \rv{g}_k \langle T_k, u^{\otimes r} \rangle } 
    \\
    &=  \E \max \bigg\{ \sup_{\norm{u}_p \leq 1}  \sum_{k=1}^m \rv{g}_k \langle T_k, u^{\otimes r} \rangle  , \sup_{\norm{u}_p \leq 1}  \sum_{k=1}^m -\rv{g}_k \langle T_k, u^{\otimes r} \rangle  \bigg \}
    \\
    &\leq \E \sup_{\norm{u}_p \leq 1}  \sum_{k=1}^m \rv{g}_k \langle T_k, u^{\otimes r} \rangle  + \E \sup_{\norm{u}_p \leq 1}  \sum_{k=1}^m -\rv{g}_k \langle T_k, u^{\otimes r} \rangle 
    \\
    &= 2 \E \sup_{\norm{u}_p \leq 1}  \sum_{k=1}^m \rv{g}_k \langle T_k, u^{\otimes r} \rangle . 
    \end{split}
    \end{equation*}
    The expected supremum of the process in the last equality can be bounded by plugging the estimate from \Cref{cor:pdlogcovernumberspnorm} into \Cref{thm:entropyintegral}. 
    $$ \int_0^{\infty} \sqrt{\log(\mathcal{N}(\pdball,\md,s))} \, ds \leq \int_0^{r \pvar{0}} \sqrt{d\log(3r \pvar{0}/s)}\, ds . $$
    We substitute $s' = \frac{s}{r \pvar{0}}$, which results in 
    $$ \int_0^{r \pvar{0}} \sqrt{d\log(3r \pvar{0}/s)}\, ds = r \sqrt{d} \pvar{0} \int_0^{1} \sqrt{\log(3/s')}\, ds' \lesssim r \sqrt{d} \pvar{0} .$$
\end{proof}

\begin{par}
    In the i.i.d. entry case this ``classical" bound performs well, as the next result will show. The parameter $\pvar{0}$ in some sense nicely exploits the orthogonality of tensors for this model. We also prove optimality of the exponent in \Cref{cor:normtransition} with this example.
\end{par}

\begin{lemma}\label{lemma:optimalitynormtransition}
    For every $(i_1, \ldots, i_r) \in [d]^r$ with $i_1 \leq \ldots \leq i_r$, we have an independent standard Gaussian $\rv{g}_{(i_1, \ldots, i_r)}$ and define the symmetric tensor $T_{(i_1, \ldots, i_r)}$  with 
    $$ T_{(i_1, \ldots, i_r)}[e_{i_{\tau(1)}}, \ldots, e_{i_{\tau(r)}}] = 1$$
    for every permutation $\tau \in S_r$ and $0$ otherwise. The symmetric random tensor 
    $$ T = \sum_{1 \leq i_1, \ldots, i_r \leq d} \rv{g}_{(i_1, \ldots, i_r)} T_{(i_1, \ldots, i_r)}$$
    has the following properties for $2 \leq p \leq 2r$:
    $$ d^{\frac{r}{2}} \ceq{r} \E \normf{T} \ceq{r} d^{\frac{r}{p}-\frac{1}{2}}\E \norminp{T}  .$$
\end{lemma}

\begin{proof}
     We use \Cref{prop:trivialsigmastarbound} to upper bound the expected injective $\ell_p$ norm of $T$, which requires an estimate for $\pvar{0}$. Let $u_1,\ldots, u_r \in \pdball$. Since the tensors $T_{(i_1, \ldots, i_r)}$ are orthogonal and have Frobenius norm at most $\sqrt{r!}$, we have
    $$ \sum_{1 \leq i_1, \ldots, i_r \leq d} \langle T_{(i_1, \ldots, i_r)}, u_1 \otimes \cdots \otimes u_r \rangle^2 \leq r! \normf{u_1 \otimes \cdots \otimes u_r}^2 = r! \sum_{i_1, \ldots, i_r =1}^d(u_1)_{i_1}^2 \cdots (u_r)_{i_r}^2 . $$
    Taking the supremum over $u_1, \ldots, u_r$ being in $\pdball$ we get 
    $$ \pvar{0}^2 \leq r! \sup_{u_1, \ldots, u_r \in \pdball} \sum_{i_1, \ldots, i_r =1}^d(u_1)_{i_1}^2 \cdots (u_r)_{i_r}^2 \leq r! \sup_{u_1, \ldots, u_r \in \pdball} \norm{u_1}_2^2 \cdots \norm{u_r}_2^2 \leq r! d^{r-\frac{2r}{p}} .$$
    The last inequality follows from using Hölder's inequality ($\norm{u_q}_2 \leq d^{\frac{1}{2}-\frac{1}{p}}\norm{u_q}_p$). Thus,
    $$ \E \norminp{T} \cleq{r} \sqrt{d}\pvar{0} \cleq{r} d^{\frac{r+1}{2}-\frac{r}{p}}.$$
    We show that the reverse inequality also holds. We lower bound the expected Frobenius norm of $T$. If every Gaussian $\rv{g}_{(i_1, \ldots, i_r)}$ was replaced by an independent rademacher random variable $\boldsymbol{\eps}_{(i_1, \ldots, i_r)}$ (which is $1$ with probability $1/2$ and $-1$ with probability $1/2$), then the Frobenius norm of the tensor would be equal to $d^{r/2}$ almost surely, as every entry has absolute value $1$. The expected norm of the Gaussian tensor dominates the expected norm of the rademacher tensor, which is a classical fact that holds more generally in Banach spaces (see section 4.2 in~\cite{LT91}). Therefore,
    $$ \E \Big\|\sum_{1 \leq i_1, \ldots, i_r \leq d} \rv{g}_{(i_1, \ldots, i_r)} T_{(i_1, \ldots, i_r)}\Big\|_{\mathrm{F}} \gtrsim \E \Big\|\sum_{1 \leq i_1, \ldots, i_r \leq d} \boldsymbol{\eps}_{(i_1, \ldots, i_r)} T_{(i_1, \ldots, i_r)}\Big\|_{\mathrm{F}} = d^{\frac{r}{2}} .$$
    We can use \Cref{cor:normtransition} to also get a upper bound for the Frobenius norm using the injective $\ell_p$-norm.
    $$d^{\frac{r}{2}} \lesssim \E \normf{T} \cleq{r} d^{\frac{r}{p}-\frac{1}{2}}\E \norminp{T} \cleq{r} d^{\frac{r}{p}-\frac{1}{2}} d^{\frac{r+1}{2}-\frac{r}{p}} = d^{\frac{r}{2}} .$$
    This implies that all inequalities here are actually equivalences, (up to constants that may depend on $r$) which finishes the proof.
\end{proof}

\begin{par}
    While the bound from \Cref{prop:trivialsigmastarbound} captures the correct dimensional dependence for the i.i.d. entry model, it performs rather poorly for most other models, especially when $\pvar{0} \asymp \typetwopvar$ holds, which is the case for example when $T_1 = \ldots = T_n$. To further illustrate the shortcomings of this bound, consider the case $n =d$, where we have tensors $T_1, \ldots, T_d \in (\R^d)^{\otimes r}$ with $\norminp{T_k} \leq 1$ for all $1 \leq k \leq d$. The best bound for $\pvar{0}$ in terms of $d$ in this case is $\pvar{0} \leq \typetwopvar \leq \sqrt{d}$. 
    Thus, \Cref{prop:trivialsigmastarbound} provides the estimate
    $$ \E \Big\|\sum_{k=1}^n \rv{g}_k T_k\Big\|_{\I_p} \cleq{r} d.$$
    This does not improve over the trivial triangle inequality:
    $$ \E \Big\|\sum_{k=1}^d \rv{g}_k T_k\Big\|_{\I_p} \leq \sum_{k=1}^d \E \norminp{\rv{g}_k T_k} \lesssim d .$$
    In contrast, our result in \Cref{cor:boundtype2constant} does not have this drawback and gives the bound
    $$ \E \Big\|\sum_{k=1}^n \rv{g}_k T_k\Big\|_{\I_p} \cleq{r,p} d^{1- \frac{1}{\max \{p,2r \}}}.$$
\end{par}

\section{Proof of Main Results}
\label{sec:proof}

\begin{par}
   The goal of this section is to first prove \Cref{thm:masterthm}, and then  \Cref{cor:boundtype2constant} as a consequence. More specifically, most of the effort will go into proving
   $$ \E \bigg[ \Big\| \sum_{k=1}^n \rv{g}_k  T_k \Big\|_{\I_p}\bigg] \cleq{r} \sqrt{p}d^{\frac{1}{2}-\frac{1}{p}}\left(\log(d+1)\pvar{1} + \max_{2 \leq q \leq r} \Big \{ \pvar{0}^{1-\frac{1}{q}} \pvar{q}^{\frac{1}{q}} \Big \} \right).$$
   The proof will be divided into three parts (\Cref{subsection:probdistestimates} - \ref{subsection:symbed} respectively). The first two parts will prove this bound for symmetric, diagonal-free tensors with a symmetric version of the $\ell_p$ injective norm. The first part deals with estimations involving the natural distance $\md$, where diagonal-freeness is used to better estimates in terms of distance $\md$. The second part then proves the covering number bound and computes the entropy integral. The third part is devoted to removing the symmetry and diagonal-free assumptions using symmetric embeddings. 
   Finally, the lower bound in \Cref{cor:boundtype2constant} will be proved in \Cref{subsec:proof_lower_bound}. 
\end{par}

\subsection{Probabilistic Distance Estimates and Diagonal-Free Tensors}\label{subsection:probdistestimates}

\begin{par}
    As seen in \Cref{sec:technique}, we will have to upper bound the expected distance between a point and a random point in an $\ell_p$-ball around it. A helpful step in this computation is to estimate the second moment of a tensor evaluated at a random point in $\pdball$. For some tensors (like the identity matrix in the case $p=2$) this quantity may be of the same order as its norm, but this issue can be resolved by assuming diagonal-freeness.
\end{par}

\begin{defn}[Diagonal-freeness]
    An order $r$ tensor $T$ is called~\emph{diagonal-free}, if $T_{i_1, \ldots, i_r}=0$ when some index appears at least twice. In other words, $T_{i_1, \ldots, i_r}=0$, if there exist $a \neq b$, such that $i_a = i_b$.
\end{defn}

\begin{par}
    As we will see in \Cref{subsection:symbed}, the symmetric embedding of a tensor is diagonal-free, so this assumption is without loss of generality. 
\end{par}

\begin{lemma}\label{lemma:ballcoordmultiplication}
    Let $\rv{b} \in \pdball$ be a random vector distributed according to the normalized Lebesgue measure on $\pdball$ and let $r \leq d$ be an integer and $2 \leq p < \infty$, then
    $$ \E[\rv{b}_1^2 \cdots \rv{b}_r^2] \leq d^{-\frac{2r}{p}}.$$
\end{lemma}

\begin{proof}
    It suffices to show 
    $$ \E[\abs{\rv{b}_{1}}^p \cdots \abs{\rv{b}_{r}}^p] \leq d^{-r}, $$
    as then the desired statement follows from Jensen's inequality.
    $$ \E[\rv{b}_{1}^2 \cdots \rv{b}_{r}^2] \leq \E[\abs{\rv{b}_{1}}^p \cdots \abs{\rv{b}_{r}}^p]^{2/p} \leq  d^{-\frac{2r}{p}} $$
    We use induction over $r$. For $r=1$ we have $\E [\abs{\rv{b}_1}^p] \leq d^{-1}$ since 
    $$  d \E [\abs{\rv{b}_1}^p] = \sum_{j=1}^d \E [\abs{\rv{b}_j}^p] \leq 1.$$
    Assume the statement holds for $r-1$, then because $\norm{\rv{b}}_p \leq 1$ almost surely, we get
    $$ d^{-r-1} \geq \E[\abs{\rv{b}_{1}}^p \cdots \abs{\rv{b}_{r-1}}^p \cdot 1] \geq \sum_{j=1}^d \E[\abs{\rv{b}_{1}}^p \cdots \abs{\rv{b}_{r-1}}^p \abs{\rv{b}_{j}}^p].$$
    If we manage to show
    $$ \E[\abs{\rv{b}_{1}}^p \cdots \abs{\rv{b}_{r-2}}^p \abs{\rv{b}_{r-1}}^{2p}]\geq  \E[\abs{\rv{b}_{1}}^p \cdots \abs{\rv{b}_{r}}^p ],$$
    we can use the symmetry of the coordinates to finish the inductive proof. The inequality above immediately follows from coordinate symmetry when we look at the inequality
    $$ \E[\abs{\rv{b}_{1}}^p \cdots \abs{\rv{b}_{r-2}}^p (\abs{\rv{b}_{r-1}}^{p}-\abs{\rv{b}_{r}}^{p})^2]\geq 0.$$
\end{proof}

\begin{lemma}\label{lemma:expectedpnormcontraction}
    Let $T$ be a symmetric diagonal-free order $r$ tensor of dimension $d$ and let $\rv{b} \in \pdball$ be a random vector distributed according to the normalized Lebesgue measure on $\pdball$ for $2 \leq p < \infty$. We have
    $$ \E[\langle T, \rv{b}^{\otimes r} \rangle^2] \cleq{r}  d^{{-2r/p}}  \normf{T}^2.$$
\end{lemma}

\begin{proof}
    In the case $d<r$ this is indeed true by the Cauchy-Schwarz inequality $\E[\langle T, \rv{b}^{\otimes r} \rangle^2] \leq \normf{T}^2 \E \normf{\rv{b}^{\otimes r}}^2 \lesssim_r \normf{T}^2$. Now assume $d \geq r$. First note that
    $$ \E[\langle T, \rv{b}^{\otimes r} \rangle^2] = \sum_{n_1,\ldots,n_r=1}^d \sum_{m_1, \ldots,m_r=1}^d\hspace{-4mm} \E[(T)_{(n_1,\ldots,n_r)}(T)_{(m_1, \ldots,m_r)} \rv{b}_{n_1} \cdots \rv{b}_{n_r} \rv{b}_{m_1} \cdots \rv{b}_{m_r}].$$
    The terms being summed can only be nonzero, when $\{n_1,\ldots,n_k\} = \{m_1,\ldots,m_k\}$, since both of these sets need to have cardinality $r$ (due to $T$ being diagonal-free) and if one index (let's say $n_1$) is only present in one of the sets, then the expectation becomes zero due to a symmetry argument ($\rv{b}_{n_1}$ can be replaced by $-\rv{b}_{n_1}$ without changing the distribution of $\rv{b}$). We can use the symmetry of $T$ to simplify the sum over the indices $m_1,\ldots, m_d$ as follows:
    $$ \E[\langle T, \rv{b}^{\otimes r} \rangle^2] =  \sum_{\substack{n_1,\ldots,n_r  \\ n_i \neq n_j \forall i \neq j }}  r! \E[(T)_{(n_1,\ldots,n_r)}^2 \rv{b}_{n_1}^2 \cdots \rv{b}_{n_r}^2] = \E[\rv{b}_{1}^2 \cdots \rv{b}_{r}^2] \hspace{-3mm} \sum_{\substack{n_1,\ldots,n_r  \\ n_i \neq n_j \forall i \neq j }} \hspace{-3mm} r! (T)_{(n_1,\ldots,n_r)}^2 $$
    Since $T$ is diagonal-free, we have
    $$ \sum_{\substack{n_1,\ldots,n_r  \\ n_i \neq n_j \forall i \neq j }} \hspace{-3mm} r! (T)_{(n_1,\ldots,n_r)}^2  = \sum_{\substack{n_1,\ldots,n_r }} \hspace{-3mm} r! (T)_{(n_1,\ldots,n_r)}^2 = r! \normf{T}^2. $$
    Combining these equalities with \Cref{lemma:ballcoordmultiplication} proves the desired statement.
\end{proof}

\begin{par}
    This estimate allows us to bound the expected natural distance between a point $x$ and a random point in a scaled $\ell_p$-ball centered at $x$. 
\end{par}

\begin{cor}\label{cor:expectedpdistance}
    Let $T_1, \ldots, T_n \in (\R^d)^{\otimes r}$ be symmetric diagonal-free tensors. Let $\rv{b} \in \pdball$ be a random vector distributed according to the normalized Lebesgue measure on $\pdball$ for $2 \leq p < \infty$ and $\rv{y} \coloneqq x + t\rv{b}$ for some $t>0$ and $x \in \pdball$. We have 
    $$ \E_{\rv{b}}[ \md(x, \rv{y}) ]  \cleq{r} \max_{1 \leq q \leq r} \left \{d^{-q/p} t^{q} \pvar{q} \right \} .$$
    
\end{cor}

\begin{proof}
     By an application of Jensen's inequality, it suffices to show
    $$ \E_{\rv{b}}[ \md(x, \rv{y})^2] \cleq{r}  \max_{1 \leq q \leq r} \left \{d^{-2q/p} t^{2q} \textstyle \sum_{k=1}^n \normf{T_k x^{\otimes (r-q)}}^2 \right \},$$
    since by \Cref{lemma:variances} the inequality $\sum_{k=1}^n \normf{T_k x^{\otimes (r-q)}}^2 \leq \pvar{q}^2$ holds. We first consider the case of a single tensor $T_k$. Let
    $$ D_k \coloneqq  \E_{\rv{b}}[ (\langle T_k, x^{\otimes r} \rangle - \langle T_k, \rv{y}^{\otimes r} \rangle)^2] = \E_{\rv{b}}[ (\langle T_k, x^{\otimes r} \rangle - \langle T_k, (x + t\rv{b})^{\otimes r} \rangle)^2] .$$
    Using the symmetry of $T_k$, we can expand the expression above as follows:
    $$ D_k = \E_{\rv{b}}[ ( \textstyle \sum_{q=1}^r \binom{r}{q} t^q \langle T_k, x^{\otimes r-q} \otimes \rv{b}^{\otimes q} \rangle)^2] = \E_{\rv{b}}[ ( \textstyle \sum_{q=1}^r \binom{r}{q} t^q \langle T_k x^{\otimes r-q},   \rv{b}^{\otimes q} \rangle)^2] .$$
    We will use the Cauchy-Schwarz inequality to upper bound $D_k$ by a sum of squares.
    $$ D_k \cleq{r} \textstyle \sum_{q=1}^r t^{2q} \E_\rv{b}[\langle T_k x^{\otimes r-q},   \rv{b}^{\otimes q} \rangle^2] .$$
    The tensor $T_k x^{\otimes r-q}$ is a diagonal-free order $q$ tensor, which allows us to use \Cref{lemma:expectedpnormcontraction} to upper bound $\E_\rv{b}[\langle T_k x^{\otimes r-q},   \rv{b}^{\otimes q} \rangle^2]$.
    $$ D_k \cleq{r} \textstyle \sum_{q=1}^r t^{2q} \E_\rv{b}[\langle T_k x^{\otimes r-q},   \rv{b}^{\otimes q} \rangle^2] \cleq{r}   \textstyle \sum_{q=1}^r t^{2q}d^{-2q/p} \normf{T_k x^{\otimes r-q}}^2 .$$
    Summing this bound over $k$ and bounding the sum over the $q$ by their maximum gives
    $$ \E_{\rv{b}}[ \md(x, \rv{y})^2] = \sum_{k=1}^n D_k \cleq{r}   \max_{1 \leq q \leq r} \left \{d^{-2q/p} t^{2q} \textstyle \sum_{k=1}^n \normf{T_k x^{\otimes (r-q)}}^2 \right \} . $$
\end{proof}

\subsection{Covering Number Bound}

\begin{par}
    This subsection contains the key part of the proof. We will summarize in short again how the covering numbers are going to be bounded. We take a maximal $2\eps$-seperated set, (which is also a $2\eps$-covering) of $\pdball$ with respect to the natural distance. The $\eps$-balls around the points in the seperated set are all disjoint, which means their intersection scaled $\ell_p$-balls around their centers are also disjoint. We can intersect all of these sets with a larger $\ell_p$-ball centered at $0$, while still retaining a constant fraction of their original volume, so by estimating the volume of these intersections we can bound the maximal number of such sets, that could fit in the larger $\ell_p$-ball.
\end{par}

\begin{par}
    By the triangle inequality, the ball $x + t \pdball$ is contained in $(1+ t) \pdball$. However, using a ball of this radius to perform our volume argument will be too loose. Instead we will show that at least half of the ball $x + t \pdball$ is contained in $\sqrt{p-1+ t^2} \pdball$.  This statement is a direct consequence of $2$-uniform convexity in $L_p$ spaces. It has been extended to Schatten classes in \cite{LBC94} (Theorem 1), here we state a much weaker version of the result.
\end{par}

\begin{thm}[2-uniform convexity]\label{thm:weakparallelogram}
    Let $x,y \in \R^d$ and $p \geq 2$. It holds
    $$ \bigg( \frac{\norm{x-y}_p^p + \norm{x+y}_p^p}{2} \bigg)^{2/p} \leq \norm{x}_p^2 + (p-1)\norm{y}_p^2 .$$
\end{thm}

\begin{cor}\label{cor:halvinglpball}
    Let $t>0$, $2 \leq p < \infty$, $x \in \pdball$ deterministic and $\rv{b} \in \pdball$ uniformly distributed according to the normalized Lebesgue measure. We have 
    $$ \Prob[\norm{x + t\rv{b}}_p \leq \sqrt{t^2 + p-1}] \geq \frac{1}{2}.$$
\end{cor}

\begin{proof}
    By symmetry of $\rv{b}$, the following random variables are identically distributed:
    $$ \norm{x + t\rv{b}}_p \sim \norm{x + t \boldsymbol{\eps} \rv{b}}_p \sim \norm{\boldsymbol{\eps}x + t\rv{b}}_p $$
    Here $\boldsymbol{\eps}$ is a rademacher random variable ($\boldsymbol{\eps}= \pm 1$ with probability $1/2$ each) which is independent from $\rv{b}$. Conditioning on $\rv{b} = b$ for any fixed $b \in \pdball$, we get
    $$ \Prob\big[\norm{\boldsymbol{\eps}x + tb}_p = \min \big\{ \norm{x + tb}_p , \norm{-x + tb}_p \big\} \Big] = \frac{1}{2}.$$
    Using \Cref{thm:weakparallelogram}, we see that 
    $$ \min \big\{ \norm{x + tb}_p , \norm{-x + tb}_p \big\} \leq \bigg( \frac{\norm{tb-x}_p^p + \norm{tb+x}_p^p}{2} \bigg)^{1/p} \leq \sqrt{\norm{tb}_p^2 + (p-1)\norm{x}_p^2} . $$
    Since $b$ and $x$ are in $\pdball$, we get that for any fixed outcome $b \in \pdball$, 
    $$ \Prob \big[\norm{\boldsymbol{\eps}x + tb}_p \leq \sqrt{t^2 + p-1} \big] \geq \frac{1}{2}.$$
    Integrating this conditional probability w.r.t. the distribution of $\rv{b}$ finishes the proof.
\end{proof}

\begin{par}
    This improvement is crucial, as large values of $t$ contribute more to the entropy integral. We now have all ingredients for our volumetric estimate the covering numbers.
\end{par}

\begin{lemma}[Covering number estimate]\label{lemma:pdvolumeargument}
    Let $T_1, \ldots, T_n \in (\R^d)^{\otimes r}$ be symmetric diagonal-free tensors and $2 \leq p < \infty$. For any $t>0$ we have
    $$\mathcal{N}(\mathbb{B}_p^d, \md, C(r) \textstyle \max_{1 \leq q \leq r} \left \{d^{-q/p} t^{q} \pvar{q} \right \})\leq 4e^{\frac{(p-1)d}{2t^2}}, $$
    where $C(r)>0$ is a constant only depending on $r$.
\end{lemma}

\begin{proof}
    Let $x \in \pdball$ and let $\rv{y} \coloneqq x + t\rv{b}$ be a uniformly distributed random vector in $x + t\pdball$ as in \Cref{cor:halvinglpball}. We estimated the expected distance between $x$ and $\rv{y}$ in \Cref{cor:expectedpdistance}. There exists a constant $C(r)>0$ depending on $r$, such that
    $$  M \coloneqq C(r)  \max_{1 \leq q \leq r} \left \{d^{-q/p} t^{q} \pvar{q} \right \} \geq \E_{\rv{b}}[ \md(x, \rv{y}) ]$$
    holds. Thus, an application of Markov's inequality yields:
    $$ \Prob_{\rv{b}} \left[ \md(x, \rv{y}) \geq 4M \right] \leq \frac{1}{4} .$$
    \Cref{cor:halvinglpball} gives the estimate
    $$ \Prob[\norm{\rv{y}}_p > \sqrt{t^2 + p-1}] \leq \frac{1}{2}.$$
    By union bounding the probability of either of these events occurring, we get the following lower bound on the complement:
    \[ \Prob[\norm{\rv{y}}_p \leq \sqrt{t^2 + p-1}, \: \md(x, \rv{y}) < 4M] \geq \frac{1}{4} .\]
    In other words, at least a ``quarter" (of the Lebesgue measure) of $x + t\pdball$ has natural distance at most $4M$ from $x$ and $\ell_p$-norm at most $\sqrt{p-1+t^2}$. 
    Now let $x_1, \ldots, x_N$ be a maximal $8M$-separated set in $\pdball$ with respect to $\md$. This set forms a $8M$-covering of $\pdball$ with respect to $\md$, as otherwise there would be a point in $\pdball$ that has distance at least $8M$ from $x_1, \ldots, x_N$, which contradicts maximality. The open balls $\operatorname{B}_{\md}(x_i, 4M)$, $1 \leq i \leq N$ must be disjoint, hence the sets 
    $$G_i \coloneqq \operatorname{B}_{\md} (x_i, 4M) \cap \sqrt{p-1+t^2}\pdball \cap (x_i+ t\pdball)$$
    are also disjoint. By the computation at the start of our proof we have 
    $$\vol(G_i) \geq \frac{1}{4}\vol(x_i+ t\pdball) = \frac{t^d}{4}\vol(\pdball).$$
    The union of all $G_i$ is also contained in $\sqrt{p-1+t^2}\pdball$, therefore, 
    $$ N\frac{t^d}{4}\vol(\pdball) \leq \sum_{i=1}^N \vol(G_i) \leq \vol(\sqrt{p-1+t^2}\pdball) = (p-1+t^2)^{d/2}\vol(\pdball),$$
    which implies the desired covering number bound 
    $$ \mathcal{N}(\pdball, \md, 8M) \leq N \leq 4 \left( \frac{p-1+t^2}{t^2} \right)^{d/2} \leq 4e^{\frac{(p-1)d}{2t^2}}.$$
\end{proof}

\begin{rem}\label{rem:latalareference}
    While preparing this manuscript, it was brought to our attention\footnote{We thank Ramon van Handel for pointing out the reference~\cite{Lat06}.} 
    that an essentially equivalent version of this covering number estimate has already been attained in the special case $p=2$ in~\cite[Corollary 2]{Lat06}. In contrast to our techniques, Latała's argument uses an inductive procedure over the tensor order, combined with Gaussian measure estimates. 

    The bound was later used to obtain a different tensor concentration inequality for the special case of $p=2$ using chaining machinery  (Theorem 2 and Theorem 3 in \cite{Lat06}). This bound does not have the logarithmic factor in \Cref{thm:masterthm} but comes with additional terms which are too large for the purpose of estimating type constants. For the application in~\cite{Lat06} the bound sufficed, as the dimension-freeness of the bound was one of the main goals in Latała's paper. Although it is not the focus of the present manuscript, it might further be possible to use Latała's chaining argument to obtain an analogue of \Cref{thm:masterthm} without logarithmic factors.

\end{rem}

\begin{cor}\label{cor:pdlogcovernumbers}
    Let $T_1, \ldots, T_n \in (\R^d)^{\otimes r}$ be symmetric diagonal-free tensors and $2 \leq p < \infty$. For any $s>0$ we have
    $$ \sqrt{\log(\mathcal{N}(\mathbb{B}_p^d, \md, s))} \cleq{r} 1 + \sqrt{p} d^{\frac{1}{2}-\frac{1}{p}} \max_{1 \leq q \leq r} \bigg \{ s^{-\frac{1}{q}} \pvar{q}^{\frac{1}{q}} \bigg \} $$
\end{cor}

\begin{proof}
    Let $s > 0$ and pick $t>0$, such that $s= C(r) \textstyle \max_{1 \leq q \leq r} \left \{d^{-q/p} t^{q} \pvar{q} \right \}$ holds, with $C(r)$ being the constant in \Cref{lemma:pdvolumeargument}. (Such a $t$ always exists by continuity of the previous expression with respect to $t$, if any of the variances is nonzero.) We claim
    $$ \frac{\sqrt{d}}{t} \cleq{r} d^{\frac{1}{2}-\frac{1}{p}} \max_{1 \leq q \leq r} \bigg \{ s^{-\frac{1}{q}}\pvar{q}^{\frac{1}{q}} \bigg \}$$
    holds, which combined with \Cref{lemma:pdvolumeargument} yields the desired statement. To prove our claim, let $q$ be the index, such that
    $$ s= C(r) d^{-q/p} t^{q} \pvar{q} $$
    holds. Rearranging this equality shows
    $$ \frac{\sqrt{d}}{t} =  C(r)^{\frac{1}{q}}s^{-\frac{1}{q}}d^{\frac{1}{2}-\frac{1}{p}} \pvar{q}^{\frac{1}{q}}.$$
    The claim follows by taking the maximum on the right side over all indices $1\leq q \leq r$.
\end{proof}

\begin{par}
    The estimate from \Cref{cor:pdlogcovernumbers} is problematic for small $s$, since the entropy integral would in that case diverge. Fortunately, we can replace the estimate by the estimate from \Cref{cor:pdlogcovernumberspnorm}. It remains to integrate these bounds using Dudley's entropy integral. The main result for the norm of arbitrary tensors will follow in the next section from a technical trick using the symmetric embedding. 
\end{par}

\begin{thm}\label{thm:simplifiedpdbound}
    Let $T_1, \ldots, T_n \in (\R^d)^{\otimes r}$ be symmetric diagonal-free tensors and $2 \leq p < \infty$. If $\rv{g}_1, \ldots, \rv{g}_n$ are i.i.d. standard Gaussians, then
    $$ \E \left[ \sup_{u \in \pdball}  \sum_{k=1}^n \rv{g}_k \langle T_k, u^{\otimes r} \rangle \right] \cleq{r} \sqrt{p}d^{\frac{1}{2}-\frac{1}{p}}\bigg(\log(d+1)\pvar{1} + \max_{2 \leq q \leq r} \left \{ \pvar{0}^{1-\frac{1}{q}} \pvar{q}^{\frac{1}{q}} \bigg \} \right).$$
\end{thm}

\begin{proof}
    Let $\Delta$ be the diameter of $\pdball$ with respect to $\md$. Dudley's entropy integral gives:
    $$ \E \bigg[ \sup_{u \in \edball}  \sum_{k=1}^n \rv{g}_k \langle T_k, u^{\otimes r} \rangle \bigg] \lesssim \int_0^{\xi} \sqrt{\log(\mathcal{N}(\pdball, \md, s))} \, ds + \int_\xi^{\Delta} \sqrt{\log(\mathcal{N}(\pdball, \md, s))} \, ds$$
    The integral from $0$ to $\xi$ will use the estimate from \Cref{cor:pdlogcovernumberspnorm}. We substitute $s' = s/\xi$.
    $$ \int_0^{\xi}\sqrt{d\log(3r \pvar{0}/s)} \, ds = \int_0^{1}\xi \sqrt{d\log(3r \pvar{0}/(s'\xi))} \, ds' \lesssim \xi \sqrt{d\log(3r \pvar{0}/\xi)} + \xi \sqrt{d} $$
    Setting $\xi = \pvar{0}/\sqrt{d}$ yields
    $$ \int_0^{\xi} \sqrt{\log(\mathcal{N}(\pdball, \md, s))} \, ds \cleq{r} \pvar{0}\sqrt{\log(d+1)}.$$
    The integral from $\xi$ to $\Delta$ uses the bound from \Cref{cor:pdlogcovernumbers}, the maximum will be replaced by a sum and we consider each term separately. Let $2 \leq q \leq r$, then 
    $$ \sqrt{p}d^{\frac{1}{2}-\frac{1}{p}}\int_0^{\Delta} s^{-\frac{1}{q}}\pvar{q}^{\frac{1}{q}} \, ds \cleq{r} \sqrt{p} d^{\frac{1}{2}-\frac{1}{p}}\Delta^{1-\frac{1}{q}} \pvar{q}^{\frac{1}{q}}.$$
    For $q=1$ we have
    $$  \sqrt{p}d^{\frac{1}{2}-\frac{1}{p}} \int_\xi^{\Delta} s^{-1}\pvar{1} \, ds \cleq{r}   \sqrt{p}d^{\frac{1}{2}-\frac{1}{p}}\log(\Delta/ \xi ) \pvar{1}.$$
    The diameter can be upper bounded using the Cauchy-Schwarz inequality:
    $$ \Delta = \max_{u,v \in \pdball} \sqrt{\sum_{i=1}^n (\langle T_i, u^{\otimes r} \rangle - \langle T_i, v^{\otimes r} \rangle)^2 } \leq \max_{u,v \in \pdball} \sqrt{2\sum_{i=1}^n \big(\langle T_i, u^{\otimes r} \rangle^2 + \langle T_i, v^{\otimes r} \rangle^2\big) } \leq 2 \pvar{0} $$
     Moreover, $\pvar{0} \leq d^{\frac{1}{2}-\frac{1}{p}}\pvar{1}$ holds, which can be seen by applying Hölders inequality: 
     \begin{align*}
        d^{1- \frac{2}{p}} \pvar{1}^2 
        & \geq  \sup_{u_l \in \pdball, \: 1 \leq l \leq r-1, } \sum_{k=1}^n \norm{T_k [u_1, \ldots, u_{r-1}, \cdot]}_{\frac{p}{p-1}}^2 \\
        & \geq  \sup_{u_l \in \pdball, \: 1 \leq l \leq r, } \sum_{k=1}^n (T_k [u_1, \ldots, u_r])^2 = \pvar{0}^2     
     \end{align*}

     Thus, the integral from $0$ to $\xi$ is being absorbed by the other terms. Combining all upper bounds and bounding the sum by a maximum yields:
    $$ \int_0^{\Delta} \sqrt{\log(\mathcal{N}(\pdball, \md, s))} \, ds \cleq{r} \sqrt{p}\log(d+1)d^{\frac{1}{2}-\frac{1}{p}}\pvar{1} + \sqrt{p}d^{\frac{1}{2}-\frac{1}{p}}\max_{2 \leq q \leq r} \bigg \{ \pvar{0}^{1-\frac{1}{q}} \pvar{q}^{\frac{1}{q}} \bigg \} $$
\end{proof}

\begin{rem}
    It is possible to reduce the logarithmic factor in \Cref{thm:simplifiedpdbound} by choosing the splitting of the two integrals at a slightly better place. This will only make a difference in special cases and since logarithmic factors are not our highest priority, we decided not to include this improvement. 
\end{rem}

\subsection{Extending to Asymmetric Tensors via Symmetric Embeddings}\label{subsection:symbed}

\begin{par}
    The goal of this section is to remove the diagonal-freeness as well as bounding the injective norm instead of the maximum of the symmetric form, to then prove a more general version of \Cref{thm:simplifiedpdbound}. We will do so by using the symmetric embedding of tensors, which generalizes the standard approach of Hermitian dilation for matrices. The calculations for general order $r$ and order $2$ are essentially the same, just more notational difficulties are going to arise.
\end{par}

\begin{defn}
    Let $v \in \R^{rd}$ be a vector. We will define the \emph{$q$-th piece} of $v$ (with respect to the order $r$) to be the vector $v^{[q]} \in \R^d$ that has coordinates 
    $$ v_{i}^{[q]} \coloneqq v_{(q-1)d+i}.$$
     We leave out the ``with respect to $r$" if it is clear from the context. 
\end{defn}

\begin{defn}
    Let $T \in (\R^d)^{\otimes r}$ be a tensor. We define the \emph{symmetric embedding} of $T$ to be the $rd$-dimensional order $r$ tensor $\sym(T) \in (\R^{rd})^{\otimes r}$ characterized by
    $$ \sym(T)[v_1, \ldots, v_r] \coloneqq \sum_{ \tau \in S_r} T[ v_{\tau(1)}^{[1]}, \ldots, v_{\tau(r)}^{[r]}].$$
    (Here $v_{\tau(q)}^{[q]}$ denotes the $q$-th piece of the vector $v_{\tau(q)}$.)
\end{defn}

\begin{rem}
    The symmetric embedding is both symmetric and diagonal-free, the latter follows because a canonical basis vector has at most one nonzero coordinate, thus if one inputs such a vector twice, then every term in the sum over the permutations must be zero.
\end{rem}

\begin{par}
    Spectral properties of symmetric embeddings have been studied before in~\cite{RvL13}. It is not hard to be convinced of the fact that in the case $r=2$ the symmetric embedding and the hermitian dilation coincide.

    For symmetric tensors a natural question is whether it makes more sense to study one of the following two norms:
    $$ \norminp{T} = \sup_{u_1, \ldots, u_r \in \pdball} T[u_1, \ldots, u_r]  \quad \mbox{or} \quad \norm{T}_{\sym(\mathcal{I}_p)} \coloneqq \sup_{u \in \pdball} \abs{ T[u, \ldots, u]}$$
    In the case $p=2$ these norms actually coincide, which is a fact that has been reproved several times in the literature. The earliest reference we are aware of is~\cite{Kel28}. For other pairs $(r,p)$ this is generally not true, the simplest example would be the $2 \times 2$ diagonal matrix that has a $1$ and a $-1$ on the diagonal, there for any $p>2$ the norms are not the same. The norms are however equivalent up to constants that do not depend on the dimension $d$. This is a consequence of a classical application of the polarization formula for symmetric multilinear maps, see Corollary 1.6 and Proposition 1.8 in \cite{Din99}, here we state a weaker version of the result. (It is important to note that while these results were stated for complex vector spaces, they still hold for $\R^d$ by the same method of proof.)
\end{par}

\begin{prop}[Equivalence between injective norms]
\label{prop:symnormequiv}
    Let $T \in (\R^d)^{\otimes r}$ be a symmetric tensor and let $\norm{\cdot}$ be a norm on $\R^d$. The following inequalities hold:
    $$ \sup_{\norm{u} \leq 1} \abs{T[u, \ldots, u]} \leq  \sup_{\norm{u_1}, \ldots, \norm{u_r} \leq 1} \abs{T[u_1, \ldots, u_r]} \leq \frac{r^r}{r!} \sup_{\norm{u} \leq 1} \abs{T[u, \ldots, u]}. $$
\end{prop}

    The symmetric norm $\norm{\cdot}_{\sym(\mathcal{I}_p)}$ of $\sym(T)$ is connected to the injective norm of $T$ as the next lemma shows. This makes \Cref{prop:symnormequiv} for our purposes essentially irrelevant.

\begin{lemma}\label{lemma:symembedsymmax}
    Let $T \in (\R^d)^{\otimes r}$ be a tensor. We have
    $$ \sup_{u \in \pdball} \langle \sym(T), u^{\otimes r} \rangle = \norm{\sym(T)}_{\sym(\mathcal{I}_p)}  =  \frac{r!}{r^{\frac{r}{p}}}\norminp{T}.$$
\end{lemma}

\begin{proof}
    Let $u \in \mathbb{B}_p^{rd}$, then
    $$ \abs{\sym(T)[u, \ldots, u]} \leq \sum_{ \tau \in S_r} \abs{T[ u^{[1]}, \ldots, u^{[r]}]} \leq r! \norm{u^{[1]}}_p \cdots \norm{u^{[r]}}_p \norminp{T}.$$
    By the AM-GM inequality we have
    $$ r! \left( \big\|u^{[1]}\big\|_p^p \cdots \big\|u^{[r]}\big\|_p^p \right)^\frac{r}{rp}\norminp{T} \leq r! \bigg( \frac{\norm{u^{[1]}}_p^p + \cdots + \norm{u^{[r]}}_p^p}{r} \bigg)^\frac{r}{p}\norminp{T}.$$
    Therefore, since $\norm{u^{[1]}}_p^p + \cdots + \norm{u^{[r]}}_p^p = \norm{u}_p^p \leq 1$, we get
    $$ \abs{\sym(T)[u, \ldots, u]} \leq \frac{r!}{r^{\frac{r}{p}}}\norminp{T}.$$
    To finish the proof of the statement we need to show that there exists $u \in \mathbb{B}_p^{rd}$, for which this upper bound is an equality. Let $v_1, \ldots, v_r \in \pdball$ be vectors, such that 
    $$ \norminp{T} = T[v_1, \ldots, v_r].$$
    These vectors must all have $p$-norm equal to $1$ by maximality. Now for $1 \leq q \leq r$, we define $u \in \mathbb{B}_p^{rd}$ piece-wise as follows:
    $$ u^{[q]} \coloneqq r^{-1/p}v_q .$$
 The factor $r^{-1/p}$ ensures that $\norm{u}_p = 1$ holds. Multilinearity shows that this vector admits the upper bound.
    $$ \sym(T)[u, \ldots, u] = \sum_{ \tau \in S_r} T[ u^{[1]}, \ldots, u^{[r]}] = \frac{r!}{r^{\frac{r}{p}}}\norminp{T} .$$
\end{proof}

\begin{par}
    The connection given by \Cref{lemma:symembedsymmax} allows us to use \Cref{thm:simplifiedpdbound} to provide bounds for the norm of arbitrary tensors. The variance parameters would still depend on the symmetric embeddings of the tensors though. They can be expressed more conveniently in terms of the initial tensors when we assume symmetry as \Cref{lemma:sympvar} will show. This Lemma may also be extendable to asymmetric tensors by considering more parameters that appear due to the asymmetry.
\end{par}

\begin{lemma}\label{lemma:sympvar}
    Let $T_1, \ldots T_n \in (\R^d)^{\otimes r}$ be symmetric tensors. For all $0 \leq q \leq r$ we have
    $$ \Big\|\sum_{k=1}^n T_k \star_q T_k\Big\|_{\I_p}^{1/2} \cgeq{r} \Big\|\sum_{k=1}^n \sym(T_k) \star_q \sym(T_k)\Big\|_{\I_p}^{1/2}. $$
\end{lemma}

\begin{proof}
    We first analyze the Frobenius norm of $\sym(T_k)$.
    Let $u_1, \ldots u_{r-q} \in \mathbb{B}_p^{rd}$
    $$ \normf{\sym(T_k)[u_1, \ldots, u_{r-q}, \cdot, \ldots, \cdot]}^2 = \sum_{m_1, \ldots m_q = 1}^{rd} \sym(T_k)[u_1, \ldots, u_{r-q}, e_{m_1}, \ldots, e_{m_q}]^2 $$
    Since $T_k$ is symmetric, we can write
    $$ \sym(T_k)[u_1, \ldots, u_{r-q}, e_{m_1}, \ldots, e_{m_q}] = \sum_{\tau \in S_{r}}T_k[u_1^{[\tau(1)]}, \ldots, u_{r-q}^{[\tau(r-q)]}, e_{m_1}^{[\tau(r-q+1)]}, \ldots, e_{m_q}^{[\tau(r)]}] .$$
    Applying the Cauchy-Schwarz inequality to this sum yields
    $$ \sym(T_k)[u_1, \ldots, u_{r-q}, e_{m_1}, \ldots, e_{m_q}]^2 \leq \hspace{-0.5mm} r! \hspace{-1.5mm} \sum_{\tau \in S_{r}}T_k[u_1^{[\tau(1)]}, \ldots, u_{r-q}^{[\tau(r-q)]}, e_{m_1}^{[\tau(r-q+1)]}, \ldots, e_{m_q}^{[\tau(r)]}]^2.$$
    This bound will be summed over all possible tuples $(m_1, \ldots, m_q)$. Given numbers $1 \leq n_1, \ldots, n_q \leq d$ and a fixed permutation $\tau \in S_r$, there exists at most $1$ tuple $(m_1, \ldots, m_q)$ in the summation, such that
    $$ (e_{m_1}^{[\tau(r-q+1)]}, \ldots, e_{m_q}^{[\tau(r)]}) = (e_{n_1}, \ldots, e_{n_q}) $$
    holds, since the standard basis vectors have precisely one nonzero entry, which determines the tuple $(m_1, \ldots, m_q)$ uniquely. For the other tuples $(m_1, \ldots, m_d)$ that do not correspond to any choice of $1 \leq n_1, \ldots, n_q \leq d$ there must therefore exist some zero vector among $ \{ e_{m_1}^{[\tau(r-q+1)]}, \ldots, e_{m_q}^{[\tau(r)]} \}$. Thus, the following equalities holds:
    \begin{equation*}
        \begin{split}
            &\sum_{m_1, \ldots m_q = 1}^{rd} T_k[u_1^{[\tau(1)]}, \ldots, u_{r-q}^{[\tau(r-q)]}, e_{m_1}^{[\tau(r-q+1)]}, \ldots, e_{m_q}^{[\tau(r)]}]^2
            \\
            = &\sum_{n_1, \ldots n_q = 1}^{d} T_k[u_1^{[\tau(1)]}, \ldots, u_{r-q}^{[\tau(r-q)]}, e_{n_1}, \ldots, e_{n_q}]^2
            \\
            = &\normf{ T_k[u_{\tau(1)}^{[1]}, \ldots , u_{\tau(r-q)}^{[r-q]}, \cdot, \ldots, \cdot ]}^2
        \end{split}
    \end{equation*}
    Summing this over all permutations yields:
    $$ \sum_{m_1, \ldots m_q = 1}^{rd} \sym(T_k)[u_1, \ldots, u_{r-q}, e_{m_1}, \ldots, e_{m_q}]^2 \leq r!\sum_{\tau \in S_{r}} \normf{ T_k[u_{\tau(1)}^{[1]}, \ldots , u_{\tau(r-q)}^{[r-q]}, \cdot, \ldots, \cdot ]}^2 .$$
    We sum this bound over $k$.
    $$ \sum_{k=1}^n \normf{\sym(T_k)[u_1, \ldots, u_{r-q}, \cdot, \ldots, \cdot]}^2 \leq r!\sum_{\tau \in S_{r}} \sum_{k=1}^n \normf{ T_k[u_{\tau(1)}^{[1]}, \ldots , u_{\tau(r-q)}^{[r-q]}, \cdot, \ldots, \cdot ]}^2 .$$
    Taking the supremum over $u_1, \ldots, u_{r-q} \in \mathbb{B}_p^{rd}$ and exchanging the sum over the permutations with the supremum results in the inequality
    $$ \Big\|\sum_{k=1}^n \sym(T_k) \star_q \sym(T_k)\Big\|_{\I_p} \leq r!\sum_{\tau \in S_{r}} \sup_{u_1, \ldots, u_{r-q} \in \mathbb{B}_p^{rd}} \sum_{k=1}^n \normf{ T_k[u_{\tau(1)}^{[1]}, \ldots , u_{\tau(r-q)}^{[r-q]}, \cdot, \ldots, \cdot ]}^2.$$
    Since pieces of vectors do not have bigger $p$-norm than the original vector, we have that $u_{\tau(1)}^{[1]}, \ldots , u_{\tau(r-q)}^{[r-q]} \in \pdball$, and therefore
    $$ \sup_{u_1, \ldots, u_{r-q} \in \mathbb{B}_p^{rd}} \sum_{k=1}^n \normf{ T_k[u_{\tau(1)}^{[1]}, \ldots , u_{\tau(r-q)}^{[r-q]}, \cdot, \ldots, \cdot ]}^2 \leq \Big\|\sum_{k=1}^n T_k \star_q T_k\Big\|_{\I_p}. $$
    Hence, we have shown
    $$ \Big\|\sum_{k=1}^n \sym(T_k) \star_q \sym(T_k)\Big\|_{\I_p} \leq (r!)^2 \Big\|\sum_{k=1}^n T_k \star_q T_k\Big\|_{\I_p}.$$
\end{proof}

\begin{par}
    The link between the variance parameters of the embeddings and the original tensors was the last ingredient needed to prove our main result \Cref{thm:masterthm} in its full generality. We emphasize that the geometric part of the proof was already done in the previous two subsections.
\end{par}

\begin{thm}\label{thm:sharpestpnck}
    Let $T_1, \ldots, T_n \in (\mathbb{R}^d)^{\otimes r}$ be symmetric tensors and $2 \leq p < \infty$.
    Then
    $$ \E \bigg[ \Big\| \sum_{k=1}^n \rv{g}_k  T_k \Big\|_{\I_p} \bigg] \cleq{r} \sqrt{p}d^{\frac{1}{2}-\frac{1}{p}}\left(\log(d+1)\pvar{1} + \max_{2 \leq q \leq r} \left \{ \pvar{0}^{1-\frac{1}{q}} \pvar{q}^{\frac{1}{q}} \right \} \right).$$
\end{thm}

\begin{proof}
    Using \Cref{lemma:symembedsymmax} and linearity of the symmetric embedding we have
    $$ \E \bigg[ \Big\| \sum_{k=1}^n \rv{g}_k  T_k \Big\|_{\I_p} \bigg] \cleq{r} \E \bigg[ \sup_{u \in \pdball} \Big \langle \sum_{k=1}^n \rv{g}_k \sym(T_k), u^{\otimes r} \Big \rangle \bigg] .$$
    The tensors $\sym(T_k)$ are diagonal-free and symmetric. Hence, we can apply \Cref{thm:simplifiedpdbound}. To avoid confusion between the different variance parameters in this context we define
    $$ \pvar{q}^{\sym} \coloneqq \Big\|\sum_{k=1}^n \sym(T_k) \star_q \sym(T_k)\Big\|_{\I_p}^{1/2} \quad \qquad \pvar{q} \coloneqq \Big\|\sum_{k=1}^n T_k \star_q T_k\Big\|_{\I_p}^{1/2}$$
    for $0 \leq q \leq r$. \Cref{thm:simplifiedpdbound} shows:
    $$ \E \hspace{-0.5mm} \bigg[ \sup_{u \in \pdball} \Big \langle \sum_{k=1}^n \rv{g}_k \sym(T_k), u^{\otimes r} \hspace{-0.5mm} \Big \rangle \bigg] \hspace{-0.5mm} \cleq{r} \hspace{-0.5mm}\sqrt{p}d^{\frac{1}{2}-\frac{1}{p}} \hspace{-0.5mm}\left(\log(d+1)\pvar{1}^{\sym} \hspace{-0.5mm} + \hspace{-0.5mm} \max_{2 \leq q \leq r} \hspace{-0.5mm} \left \{ \pvar{0}^{\sym} \hspace{-0.1mm}^{1-\frac{1}{q}} \pvar{q}^{\sym} \hspace{-0.1mm}^{\frac{1}{q}} \right \} \right ) .$$
    For any $0 \leq q \leq r$ the variance parameter $\pvar{q}^{\sym}$ is upper bounded up to order-dependent factor by $\pvar{q}$, which was proven in \Cref{lemma:sympvar}. Thus,
    $$ \E \bigg[ \sup_{u \in \pdball} \Big \langle \sum_{k=1}^n \rv{g}_k \sym(T_k), u^{\otimes r} \Big \rangle \bigg] \cleq{r} \sqrt{p}d^{\frac{1}{2}-\frac{1}{p}}\left(\log(d+1)\pvar{1} + \max_{2 \leq q \leq r} \left \{ \pvar{0}^{1-\frac{1}{q}} \pvar{q}^{\frac{1}{q}} \right \} \right) .$$
\end{proof}

\begin{par}
    The bound given in \Cref{thm:sharpestpnck} is the sharpest one we are able to provide using our techniques (for general unspecified $T_k$). The next theorem will present the slightly more transparent \Cref{cor:boundtype2constant} involving the sum of squared injective norms (the type-2 variance). This result may yield significantly worse estimates than \Cref{thm:sharpestpnck}, but it illustrates well the worst-case performance of our bounds.
\end{par}

\begin{thm}\label{thm:typetwopnck}
    Let $T_1, \ldots, T_n \in (\mathbb{R}^d)^{\otimes r}$ be symmetric tensors and $2 \leq p < \infty$, we have
    $$ \E \bigg[ \Big\| \sum_{k=1}^n \rv{g}_k  T_k \Big\|_{\I_p} \bigg] \cleq{r} \sqrt{p}\log(d+1)d^{\frac{1}{2}-\min \{\frac{1}{p}, \frac{1}{2r}\}} \typetwopvar.$$
\end{thm}

\begin{proof}
    \Cref{cor:type2var} provides the estimates
    $$ \pvar{q}^{\frac{1}{q}} \cleq{r} d^{\max \{\frac{1}{p}- \frac{1}{2q},0 \}} \typetwopvar^{\frac{1}{q}}$$
    for $1 \leq q \leq r$. Combining these inequalities with $\pvar{0} \leq \typetwopvar$ and \Cref{thm:sharpestpnck} yields
    $$  \E \bigg[ \Big\| \sum_{k=1}^n \rv{g}_k  T_k \Big\|_{\I_p}\bigg] \cleq{r} \sqrt{p}d^{\frac{1}{2}-\frac{1}{p}}\left(\log(d+1)\typetwopvar + \max_{2 \leq q \leq r} \left \{  d^{\max \{\frac{1}{p}- \frac{1}{2q},0 \}} \typetwopvar \right \} \right).$$
    The result then follows from the fact 
    $$\frac{1}{2}-\frac{1}{p} + \max \Big \{\frac{1}{p}- \frac{1}{2q},0 \Big \} = \frac{1}{2}-\min \Big \{\frac{1}{p}, \frac{1}{2q} \Big \}.$$
\end{proof}

\subsection{Proof of Lower Bound in \Cref{cor:boundtype2constant}} \label{subsec:proof_lower_bound}

\begin{par}
    In this subsection, we prove the lower bound in \Cref{cor:boundtype2constant}. In particular, we construct an example of tensors, for which it is attained.
\end{par}

\begin{lemma}
    There exist tensors $T_1, \ldots, T_d \in (\mathbb{R}^d)^{\otimes r}$, such that the following holds:
    $$ \E \bigg[ \Big\| \sum_{k=1}^n \rv{g}_k  T_k \Big\|_{\I_p} \bigg] \cgeq{r,p} d^{\frac{1}{2}-\frac1p} \typetwopvar.$$
\end{lemma}

\begin{proof}
    We define the entries of the tensor $T_k$ by $(T_k)_{1, \ldots, 1, j} =1$ for all $1 \leq j \leq d$ and all other entries being zero. Let $\eps_1, \ldots, \eps_d \in \{-1,1\}$, the sum
    $\sum_{k=1}^d \eps_k T_k $ is essentially a column vector having $\eps_k$ as $k$-th entry. Pick vectors $v_q = e_1$ for $1 \leq q \leq r-1$ and define the last vector entry-wise $(v_r)_k = \eps_k d^{-\frac1p}$. Then $v_1, \ldots, v_r \in \pdball$ and
    $$ \Big\|\sum_{k=1}^d \eps_k T_k\Big\|_{\I_p} \geq  \sum_{k=1}^d \eps_k T_k[e_1, \ldots, e_1, v_r] = d^{1-\frac1p} = d^{\frac12-\frac1p} \typetwopvar $$
    holds, since $\typetwopvar = \sqrt{d}$ as every tensor has injective norm $1$. Taking the expectation over $\eps_k$ being random finishes the proof, since the expected supremum of Gaussian processes dominate the expected supremum of rademacher processes (see section 4.2 in~\cite{LT91}):
    $$ \E \bigg[ \Big\| \sum_{k=1}^n \rv{g}_k  T_k \Big\|_{\I_p} \bigg] \gtrsim \E \bigg[ \Big\|\sum_{k=1}^n \boldsymbol{\eps}_k  T_k \Big\|_{\I_p} \bigg] \geq d^{\frac12-\frac1p} \typetwopvar . $$
\end{proof}

\begin{rem}
    Several steps of our proof techniques generalize to $\ell_p$ injective norms for $p$ being outside of the regime we discussed. In fact, \Cref{thm:masterthm} also holds for $p= \infty$ by simply leaving out the ball halving step in \Cref{cor:halvinglpball}, which doesn't involve the constant $\sqrt{p}$ that would tend to infinity in this case. The bound does improve upon the classical $\eps$-net approach from \Cref{prop:trivialsigmastarbound}, but we decided to omit it from our results as the question about the type $2$ constant being $\asymp_r \sqrt{d}$ for this norm is already solved by the standard approach.
\end{rem}

\section{Open Problems and Further Related Directions}

\begin{par}
In this section, we mention further connections to other areas of mathematics and theoretical computer science to motivate the future study of injective norms of random tensors. In particular, we discuss open questions, unexplored directions or special cases and how these could lead to improvements of state-of-the-art results in different fields. 
\end{par}

\subsection{Locally Decodable Codes and Type Constants}\label{subsection:ldcs}

\begin{par}
    The connections between type (and cotype) constants of the $\ell_p$ injective norm and locally decodable codes have been observed in \cite{BNR12,Bri16,Gop18}. We give a brief overview of LDCs and mention their relationship with our bounds and conjectures. A (binary) {\em locally decodable code} (LDC) $\calC : \{0,1\}^n \rightarrow \{0,1\}^d$ is a mapping from any $n$-bit message $b \in \{0,1\}^n$ to a $d$-bit codeword $\calC(b) \in \{0,1\}^d$ that satisfies the following property: given an arbitrary vector $y \in \{0,1\}^d$ obtained by corrupting $\calC(b)$ in a constant fraction of its coordinates, one can recover any message bit $b_i$ with non-trivial probability by only querying $y$ in a few locations. The code is said to be {\em linear} if the mapping $\calC : \{0,1\}^n \rightarrow \{0,1\}^d$ is a linear map. Formally, LDC codes are defined as follows \cite{KT00}.
\end{par}

\begin{defn}[Locally Decodable Code]
\label{defn:LDC}
A code $\calC : \{0,1\}^n \rightarrow \{0,1\}^d$ is $(q, \delta, \eps)$-locally decodable if there exists a randomized decoding algorithm $\Dec(\cdot)$, which takes an $i \in [k]$ as input and given oracle access to some $y \in \{0,1\}^d$ satisfies the following properties: 
\begin{enumerate}
    \item[(1)] The algorithm $\Dec$ never makes more than $q$ queries to the string $y$; 
    \item[(2)] For all $b \in \{0,1\}^n$, $i \in [n]$, and $y \in \{0,1\}^d$ such that $\Delta(y, \calC(b)) \leq \delta n$, one has $\Prob[\Dec^{y}(i) = b_i] \geq \frac{1}{2} + \eps$, where $\Delta(x,y)$ is the Hamming distance between $x$ and $y$. 
\end{enumerate}
\end{defn}

LDCs are extensively studied in computer science. 
Apart from being natural objects in coding theory, locally decodable codes are closely related to many other areas of theoretical computer science, such as the PCP theorem \cite{AS98, ALM+98}, complexity theory \cite{Yek12}, worst-case to average-case reductions \cite{Tre04}, private information retrieval \cite{Yek10}, secure multiparty computation \cite{IK04}, derandomization \cite{DS05}, matrix rigidity \cite{Dvi11}, data structures \cite{CGW13}, and fault-tolerant computation \cite{Rom06}.
We refer to the excellent survey \cite{Yek12} for more background and applications of LDCs.

A central question studied in coding theory is the smallest possible blocklength $n$, or equivalently, the largest possible rate, that can be achieved by a $(q,\delta,\varepsilon)$-LDC (or $q$-LDC for short, as $\delta$, $\varepsilon$ are usually taken as some small constants). 
The case of $q = 2$ is essentially resolved by classical results: the  Hadamard code is a  $2$-LDC with blocklength $d = 2^n$ and a matching lower bound of $2^{\Omega(n)}$ is known \cite{KW04, GKST06, Bri16, Gop18}. 

For the case of $q=3$, there is still a wide gap in our understanding. The current best known lower bound is $d \geq \widetilde{O}(n^3)$ given in recent work \cite{AGKM23,HKM+24}, while the best known construction gives a linear binary $3$-LDC with $d \leq \exp(\exp(O(\sqrt{\log n \log \log n})))$ \cite{Yek08,Efr09}. For larger (but constant) $q > 3$, the construction in \cite{Efr09} gives an upper bound of $d \leq \exp(\exp(O((\log n)^{\frac{1}{\log q}} (\log \log n)^{\frac{\log q - 1}{\log q}})))$, while the best known lower bound is $d \geq n^{\frac{q}{q-2}} / \polylog(n)$ when $q$ is even and $d \geq n^{\frac{q+1}{q-1}} / \polylog(n)$ when $q$ is odd. We refer to the recent work of \cite{AGKM23,KM24,HKM+24} and references therein for more background on LDCs and the related notion of locally correctable codes (LCCs).

We now focus on connections between LDCs and random tensors. In particular, we highlight their relation to conjectured refinements of \Cref{cor:boundtype2constant}.

\begin{defn}[Normal LDC] \label{defn:normal_LDC}
A code $\mathcal{C}: \{-1,1\}^n \rightarrow \{-1,1\}^d$ is $(q,\delta, \varepsilon)$-normally locally decodable if for each $i \in [n]$, there is a $q$-uniform hypergraph matching $\H_i$ on the set $[d]$ (a collection of vertex disjoint hyperedges, where every edge has $q$ vertices) with at least $\delta d$ hyperedges such that for every $C \in \H_i$, one has $\p_{b \sim \{\pm 1\}^n}(b_i = \prod_{v \in C} \mathcal{C}(b)_v) \geq \frac{1}{2} + \varepsilon$. 
In particular, for linear codes, one has $b_i = \prod_{v \in C} \mathcal{C}(b)_v$ with probability 1. We therefore call these linear  $(q,\delta)$-normally locally decodable codes. 
\end{defn}

Using known reductions \cite{Yek12}, it is possible to turn any LDC codes into normal form. 

\begin{fact}[Reduction to LDC Normal Form \cite{Yek12}] \label{fact:reduction_normal_LDC}
Let $\mathcal{C}: \{0,1\}^n \rightarrow \{0,1\}^d$ be a $(q,\delta,\varepsilon)$-locally decodable codes, then there exists another code $\mathcal{C}': \{-1,1\}^n \rightarrow \{-1,1\}^{O(d)}$ that is $(q,\delta',\varepsilon')$-normally decodable, where $\delta' \geq \varepsilon \delta / 3 q^2 2^{q-1}$ and $\varepsilon' \geq \varepsilon / 2^{2q}$. 
\end{fact}

While we have presented the above definitions and results for general $\varepsilon$ and $\delta$, in the following we are going to consider $\varepsilon, \delta > 0$ being some small constants, and the focus lies on understanding the dependencies between $d$ and $n$.

For $p<2r$, \Cref{cor:boundtype2constant} does not have matching exponents  in the upper and lower bounds (in terms of their dependence on $d$), and it does not immediately imply any progress for LDCs. 
Nevertheless, natural conjectured strengthening of \Cref{cor:boundtype2constant} would imply improved lower bounds for the blocklength of certain $q$-LDCs. We formulate two conjectures, and discuss their implications to LDC lower bounds.

\begin{conj}\label{conj:actualtype2constant}
    Let $\mathcal{C}_{r,p}(d)$ be Type-$2$ constant as in \Cref{def:type2constant}($2 \leq p < \infty$), then 
    \begin{equation}\label{eq:weaktypeconj}
        \mathcal{C}_{r,p}(d) \cleq{r,p} d^{\frac12- \frac1{\max\{p,r\}}} (\log \, d)^{E_{r,p}}
    \end{equation}
    holds, where $E_{r,p}$ is an exponent only depending on $r$ and $p$. We cannot rule out that the following stronger inequality also holds: 
    $$ \mathcal{C}_{r,p}(d) \cleq{r,p} d^{\frac12- \frac1p} (\log \, d)^{E_{r,p}} .$$
\end{conj}

\begin{par}
    Both of the conjectured inequalities are consistent with the noncommutative Khintchine inequality~\eqref{eq:NCK}, (which shows $\CCC_{2,2} \lesssim \sqrt{\log d}$,) and extend it naturally. A proof of either of them would have interesting implications in coding theory, which we discuss in \Cref{subsection:ldcs}.
\end{par}

\begin{par}
    \Cref{conj:actualtype2constant} would imply new lower bounds for the length of $q$-normal LDCs and thus also for $q$-LDCs for odd (and fixed) $q \geq 5$. The proof essentially mimics a lower bound proof for $2$-LDCs~\cite{Bri16,Gop18}.
\end{par}

\begin{lemma}
    If \Cref{conj:actualtype2constant} is true, then for any $(q, \delta, \eps)$-normally locally decodable code $\mathcal{C}: \{-1,1\}^n \rightarrow \{-1,1\}^d$ we have
    $$ \eps \delta n^{\frac{q}{q-2}} \cleq{q} d (\log \, d)^{E_{q}'} $$
    for some exponent $E_{q}'$ only depending on $q$.
\end{lemma}

\begin{proof}
    Let $\mathcal{H}_1, \ldots, \mathcal{H}_n$ be the hypergraph matchings on the set $[d]$ of a $(q,\delta, \varepsilon)$-normal LDC $\mathcal{C}: \{-1,1\}^n \rightarrow \{-1,1\}^d$. Let $T_1, \ldots, T_n \in (\R^d)^{\otimes q}$ be the adjacency tensors of the hypergraph matchings. Let $\boldsymbol{\eps}_1, \ldots, \boldsymbol{\eps}_n$ be independent Rademacher random variables and $\mathcal{C}(\boldsymbol{\eps}) \in \{-1,1\}^d$ be the corresponding codeword, which satisfies that for any $i$ and for any $C \in \mathcal{H}_i$ we have $\Prob[\boldsymbol{\eps}_i = \prod_{v \in C} \mathcal{C}(\boldsymbol{\eps})_v] \geq \frac{1}{2} + \eps$. Since $\mathcal{H}_i$ is a hypergraph matching and its every hyperedge is represented in $T_i$ by index symmetry $q!$ times, we deduce
    $$ \E[\boldsymbol{\eps}_i \langle T_i, \mathcal{C}(\boldsymbol{\eps})^{\otimes q} \rangle ] = \sum_{C \in \mathcal{H}_i} q!\E \bigg[ \boldsymbol{\eps}_i  \prod_{v \in C} \mathcal{C}(\boldsymbol{\eps})_v \bigg ]  \geq \sum_{C \in \mathcal{H}_i} q!2\eps \geq q! 2\eps \delta d.$$
    Summing this bound over the index $i$ and dividing everything by $d$ yields
    $$ \E \bigg[ \sum_{i=1}^n \boldsymbol{\eps}_i \langle T_i, (d^{-\frac1q}\mathcal{C}(\boldsymbol{\eps}))^{\otimes q} \rangle \bigg] \geq q! 2\eps \delta n.$$
    Since the vector $d^{-\frac1q}\mathcal{C}(\boldsymbol{\eps})$ has $\ell_q$ norm $1$, we also get an upper bound using \Cref{conj:actualtype2constant}.
    $$ \E \left[ \sum_{i=1}^n \boldsymbol{\eps}_i \langle T_i, (d^{-\frac1q}\mathcal{C}(\boldsymbol{\eps}))^{\otimes q} \rangle \right] \leq \E \Big\|\sum_{i=1}^n \boldsymbol{\eps}_i T_i\Big\|_{\I_q} \cleq{q} d^{\frac12- \frac1q} (\log \, d)^{E_{q}} \sqrt{\sum_{i=1}^n \norminx{T_i}{q}^2} .$$
    We claim $\norminx{T_i}{q} \lesssim_q 1$. Indeed, let $u \in \mathbb{B}_q^d$, then we have by the AM-GM inequality
    $$ \abs{\langle T_i, u^{\otimes q} \rangle}= \bigg|\sum_{C \in \mathcal{H}_i} q!\prod_{v \in C} u_v \bigg| \leq \bigg|\sum_{C \in \mathcal{H}_i} q!\sum_{v \in C} \frac{\abs{u_v}^q}{q}\bigg| \leq (q-1)!.$$
    The last inequality follows from the fact that every coordinate $u_v$ appears at most once in the sum, since $\mathcal{H}_i$ is a hypergraph matching. The claim $\norminx{T_i}{q} \lesssim_q 1$ then follows from polarization (\Cref{prop:symnormequiv}). Combing all estimates we derived thus far gives
    $$ \eps \delta n \cleq{q} d^{\frac12- \frac1q} (\log \, d)^{E_{q}} .\sqrt{n}$$
    This can be rewritten as the upper bound $\eps \delta n^{\frac{q}{q-2}} \lesssim_q d (\log \, d)^{E_{q}'}$.
\end{proof}

\subsection{Directions for the Independent Entry Model}\label{subsection:refineindepentry}

\begin{par}
    Gaussian (and Rademacher) matrices with independent entries (but potentially different variances among the entries) is a model structure where geometric approaches have proven to be successful~\cite{BvH16,vH17,Lat24,GHLP17,APSS24}. In a recent breakthrough by Latala \cite{Lat24}, bounds on the spectral norm of random Rademacher matrices were obtained that are accurate up to a constant factor when the variances are either 0 or 1, and a $\log \log \log d$ factor in general. For $p \neq 2$, the works~\cite{GHLP17,APSS24} have made substantial progress in understanding Gaussian matrices with independent entries. Matching upper and lower bounds that are accurate up to logarithmic factors for different matrix injective norms were proven in~\cite{APSS24}. If one were to study tensor generalizations of this problem (as in \Cref{subsec:ind_ent_tensor_PCA}), this reference might be a good starting point.

    We emphasize that the independent entry assumption changes the nature of the problem. One reason is that the metric space given by the natural distance gains a susbtantial amount of symmetry: multiplying such a tensor entry-wise with any tensor that has $\pm 1$ entries leaves the distribution invariant. Moreover, we can expect much better concentration of the norm than in the general setting. Using Gaussian concentration one can show 
    $$ \Prob[\norminp{T} \geq t+ \E \norminp{T}] \leq \exp \Big({-\frac{t^2}{2\pvar{0}^2}} \Big),$$
    (see Corollary 4.14 in~\cite{BBvH23} for a quick proof in the case of $p  = r = 2$; the tensor generalization follows similarly). For $p=2$, the parameter $\evar{0}$ in the independent entry model is the largest variance of the entries in $T$, which can be shown with an orthogonality argument similar to the proof of \Cref{lemma:optimalitynormtransition}. In the general case, however, we may just get concentration of a single scalar random variable when all tensors are the same. 
    We leave the study of the independent entry model for tensors and further applications to tensor PCA as interesting open problems. 
\end{par}

\subsection{Alternative Approaches for the Euclidean Case}

\begin{par}
    \Cref{thm:masterthm} takes the following form in the case of $p = r = 2$ (the Euclidean matrix case):
    $$ \E \norminx{T}{2} \lesssim \log d \, \, \evar{1} +   \evar{2}^{\frac12} \evar{0}^{\frac{1}{2}}.$$
    The term $\evar{2}^{\frac12} \evar{0}^{\frac{1}{2}}$ is unnecessary as the noncommutative Khintchine inequality~\eqref{eq:NCK} shows. It is therefore tempting to conjecture that the second term 
    $$ \max_{2 \leq q \leq r} \pvar{q}^{\frac1q} \pvar{0}^{\frac{q-1}{q}}$$
    in~\eqref{eq:intromainthm} might not be needed when $p=2$. We were thus far unable to come up with an example where this term is required. There is further evidence why this suspicion may be correct. Consider the modified model where $T_1, \ldots, T_n \in (\R^d)^{\otimes r}$ are symmetric tensors and $Q_1, \ldots, Q_n$ are random orthogonal matrices distributed according to the Haar measure on $\operatorname{O}(d)$. \cite{Kevinsthesis} showed the following bound via a classical discretization argument:
    $$ \E \Big\|\sum_{k=1}^n T_k[Q_k\cdot, \cdot, \ldots, \cdot]\Big\|_{\I_2} \ceq{r} \Big\|\sum_{k=1}^n T_k \star_1 T_k\Big\|_{\I_2}^{1/2} = \evar{1}$$
    This model can be seen as a type of ``free analogue" of the Gaussian model we study in the Euclidean matrix case. Informally, the orthogonal matrices induce noncommutativity between the matrices, which causes the logarithmic term in the noncommutative Khintchine inequality~\eqref{eq:NCK} to disappear. A more detailed discussion on this phenomenon can be found in~\cite{BBvH23}. The fact that such a generalization is also true for tensors leads us to conjecture that only the parameter $\evar{1}$ is relevant for controlling the injective $\ell_2$ norm.

    \begin{conj}\label{conj:Euclideantensornck}
    Let $T \in (\R^d)^{\otimes r}$ be a symmetric random jointly Gaussian tensor with $\E[T]=0$, there exists an exponent $E(r)>0$ only depending on $r$, such that 
    $$ \E \norminx{T}{2} \cleq{r} (\log d)^{E(r)} \evar{1} .$$
    \end{conj}

    This estimate would actually be two-sided, since one can also show by an application of \Cref{lemma:variances} and Jensen's inequality that
    $$ \E \norminx{T}{2}^2=  \E \left [ \sup_{u_1, \ldots, u_{r-1} \in \edball} \normf{T[u_1, \ldots, u_{r-1}, \cdot]}^2 \right ]= \E \norminx{T \star_1 T }{2} \geq \evar{1}^2 ,$$
    which would imply $\E \norminx{T}{2} \gtrsim \evar{1}$, since we have $\E \norminx{T}{2}^2 \asymp (\E \norminx{T}{2})^2$, this is a fact that holds more generally for Gaussian vectors in banach spaces (see section 4.3 in~\cite{LT91}). We remark that \Cref{conj:Euclideantensornck} requires the stronger version of our type constant conjecture in \Cref{conj:actualtype2constant} to be true for the case $p=2$. 
    
    Restricting to $\ell_2$ injective norms may allow for a wider range of tools to be used. For instance, tensor generalizations of the moment method (or tensor networks) have been studied recently \cite{Gur14,Gur16,Evn21,KMW24,Bon24}. There are a few difficulties that arise, however, if one were to use these techniques to attack \Cref{conj:Euclideantensornck}. One obstacle for this approach is the missing algebraic machinery for deterministic tensor networks, i.e., it is unclear how to generalize all the trace inequalities used in the proof of the noncommutative Khintchine inequality to tensor networks. 
    
    There are alternative approaches based on tensor flattening, which have already proven to be useful in practice~\cite{HSSS16,GM15,BDNY24}. 
    One could see a tensor $T \in (\R^d)^{\otimes 4}$ as a matrix in $\R^{d^2 \times d^2}$, the operator norm of this matrix would also dominate the tensor injective norm. This approach, however, disregards the higher degree structure of tensors and cannot give accurate bounds for deterministic tensors, which \Cref{conj:Euclideantensornck} also covers. There exist tensors $A \in (\R^d)^{\otimes 4}$ with $\|A\|_{\I_2} \leq 1$, but when seen as a $d^2 \times d^2$ matrix has operator norm at least $c\sqrt{d}$, incurring a $\Omega(\sqrt{d})$ factor loss. Such a tensor $A$ can be obtained by sampling its entries as independent Gaussians and then divide the entries by $C\sqrt{d}$ for a sufficiently large $C$.

    Due to such issues, more sophisticated flattening techniques have been considered in~\cite{HSSS16,GM15,BDNY24}, but the core problem remains of finding a general flattening technique from which one can trace back to the injective norm without losing dimensional factors. 
    The difficulty is likely to be fundamental, as there are conjectures about the hardness of approximating the injective $\ell_2$-norm using low-degree polynomials~\cite{HKP+17}, which would contradict the possibility of using a low power moment method to get bounds for the injective norm. Being able to use a power of order $\polylog(d)$ is necessary to prove a version of \Cref{conj:Euclideantensornck}, as otherwise the moment growth of the Gaussian random variables starts having a superpolylogarithmic contribution.
\end{par}

\appendix

\section{An Introduction for A Theoretical CS Audience}\label{sec:introTCS}

Matrix concentration inequalities have played an increasingly important role in many applications in computer science, mathematics, and other related fields. A well-studied and very general question in the study of matrix concentration inequalities is the operator norm of an arbitrary $d \times d$ symmetric\footnote{The symmetry assumption can be removed by using hermitian dilation (see \Cref{subsection:symbed}).} random matrix $X$ with mean-zero jointly Gaussian entries. Any such random matrix can be expressed as 
\begin{align} \label{eq:joint_gauss_mat}
X = \sum_{k=1}^n g_k A_k ,
\end{align}
where $g_k$'s are i.i.d. standard Gaussian variables, and $A_k$'s are (deterministic) $d \times d$ symmetric matrices. For such matrices, the classical non-commutative Khintchine (NCK) inequality \cite{LP86,LP91,Pis03} or matrix Chernoff bounds \cite{AW02,Oli10,Tro15} imply\footnote{We write $x \lesssim y$ if $x \leq Cy$ for a universal constant $C$ (see \Cref{subsec:notation_TCS}).} 
\begin{align} \label{eq:NCK_TCS}
\sigma(X) \lesssim \E \|X\|_{\op} \lesssim \sigma(X) \sqrt{\log d} ,
\end{align}
where the matrix standard deviation parameter $\sigma(X)$ is defined as
\begin{align*}
    \sigma(X)^2 := \| \E X^2 \|_{\op} = \Big\|\sum_{k=1}^n A_k^2 \Big\|_\op.
\end{align*}
Both the upper and lower bounds in \eqref{eq:NCK_TCS} are tight in general, and they can be further extended to asymmetric matrices via a standard dilation argument, and to sums of general random matrices, e.g. the matrix Bernstein inequality, either by a symmetrization argument \cite{Rud96o,Tro16}, or the general proof method in \cite{Oli10,Tro15}, or via a universality principle \cite{BvH22}. 
There have also been recent successful efforts in sharpening these bounds when the matrices $A_i$ are non-commutative \cite{BBvH23,BvH22,BCSvH24}. 

These matrix concentration inequalities have been successfully applied to a wide range of applications, such as numerical linear algebra \cite{Mah11,Woo14}, high dimensional statistics \cite{Mui09,MKT24}, compressed sensing \cite{FR13,CRPW12}, principle component analysis (PCA) \cite{ABBS14}, combinatorial optimization \cite{So11,NRV13}, discrepancy theory \cite{BJM23}, and coding theory \cite{Bri16,Gop18}. We refer to the monograph \cite{Tro15} for more details. 

\medskip
\noindent \textbf{Tensor Concentration Inequalities.} Despite their tremendous success, matrix concentration inequalities fall short in the study of many important problems, where the random objects that naturally appear in these applications are tensors in $(\R^d)^{\otimes r}$. Here, $d$ is called its {\em dimension} and $r$ is its {\em order}. 
Examples of such applications include locally decodable (or correctable) codes \cite{Yek12,Bri16,Gop18}, tensor decomposition \cite{AGH+14,GHJY15}, tensor completion \cite{JO14,MS18}, tensor PCA \cite{RM14,AMMN19}, community detection on hypergraphs \cite{KBG17,GD17,PZ21}, statistical inference in tensor ensembles \cite{KMW24}, hypergraph expanders \cite{ZZ21}, and dispersive partial differential equations \cite{BDNY24}. 
These applications raise a need to understand random tensors and their concentration properties under more general norms (beyond the natural counterpart of the matrix operator norm).\footnote{Sometimes, it is possible to apply matrix concentration inequalities to the matrices obtained by ``flatenning out'' the tensors (see \Cref{subsec:related_work}), but this approach usually relies on additional structural assumptions of the underlying problems and often don't achieve the best possible result.}

For a $d$-dimension, $r$-order tensor $T \in (\R^d)^{\otimes r}$, denote its $\ell_p$ injective norm\footnote{It is also natural to consider a symmetric version of the $\ell_p$ injective norm where the vectors $x_1, \cdots, x_r$ are forced to be the same, i.e., $\norm{T}_{\sym(\mathcal{I}_p)} \coloneqq \sup_{x \in \pdball} \abs{ T[x, \ldots, x]}$. For symmetric tensors, these two definitions are equivalent up to a factor depending only on the order $r$  (see \Cref{prop:symnormequiv}).} by
    \begin{equation}\label{eq:injnormdef_TCS}
    \norminp{T} \coloneqq \sup_{\norm{x_1}_p, \ldots, \norm{x_r}_p \leq 1} \langle T, x_1 \otimes \ldots \otimes x_r \rangle,  \end{equation}
where $d, r\geq 2$ are integers and $p\geq2$. Note that $r = 2$ and $p = 2$ correspond to the operator norm of matrices. $\ell_p$ injective norms for $p > 2$ appear in various applications, e.g. $q$-query locally decodable codes correspond to $\ell_q$ injective norms of $q$-order tensors (see \Cref{subsection:ldcs}), and have also been considered in the context of random matrices \cite{GHLP17,APSS24}.

In stark contrast to the significant progress made in studying matrix concentration inequalities for the operator norm, little is known for their tensor counterpart or for general $\ell_p$ injective norms, i.e. when either $r \neq 2$ and $p \neq 2$. 
Let us consider a mean-zero symmetric jointly Gaussian tensor, the natural analog of \eqref{eq:joint_gauss_mat}, which can be written as
\begin{align} \label{eq:Tasgaussianseries_TCS}
T = \sum_{k=1}^n g_k T_k ,
\end{align}
where $g_k \sim \mathcal{N}(0,1)$ are i.i.d. but $A_k \in (\R^d)^{\otimes r}$ are now symmetric order-$r$ tensors. 
Previously, progress in understanding \eqref{eq:Tasgaussianseries_TCS} has been made in the following special cases: 
\begin{enumerate}
    \item When $T$ has i.i.d. standard Gaussian entries (i.e., each $T_k$ corresponds to a single entry of $1$ in $T$) and $p=2$, it is well-known that $\E[\|X\|_{\ell_2}] \lesssim C_r \sqrt{d}$ for a constant $C_r$ depending on $r$ \cite{TS14,NDT10,DM24}. 
    \item When each tensor $T_k$ in \eqref{eq:Tasgaussianseries_TCS} is rank-$1$, i.e. $T_k = u_k^{\otimes r}$ for some $u_k \in \R^d$, and $p=2$, the argument based on majorizing measures in the classical work of Rudelson \cite{Rud96c} can be extended to derive an analog of the non-commutative Khintchine inequality (with a logarithmic factor in $n$ instead of $d$) \cite{Kevinsthesis}.
    \item In the different context of moment inequalities for Gaussian chaoses, a volumetric bound for the case of $p=2$ was proved by Latała~\cite{Lat06}.\footnote{While preparing this manuscript, it was brought to our attention that Latała's work \cite{Lat06} provides a tensor concentration inequality for $p=2$. While the inequality is different, it is based on an essentially equivalent covering estimate (that was proved with a rather different approach, see \Cref{rem:latalareference}). It may further be possible to obtain an analogue of \Cref{thm:masterthm_TCS} (for the specific case of $p=2$) without logarithmic factors by using the chaining argument in~\cite{Lat06}.}
\end{enumerate}

Beyond these special cases, the question of bounding the $\ell_p$ injective norm of \eqref{eq:Tasgaussianseries_TCS} beyond the trivial triangle inequality, i.e. $\E \|T\|_{\I_p} \lesssim \sum_{k=1}^n \|A_k\|_{\I_p}$, becomes elusive.
The main reason for this bottleneck is that all known proofs of \eqref{eq:NCK_TCS} use operator-theoretic tools,\footnote{The study of Schatten-$p$ norms (traces of powers) and the fact that they approximate the spectral norm well for large $p$ are key ingredients in the proof.} which are thus far unavailable when either $r \neq 2$ or $p \neq 2$.  In fact, there is no known geometric argument to establish \eqref{eq:NCK_TCS}, even for the operator norm of matrices, when the summands have rank higher than one. 
This is another core motivation of this paper, and in fact, a question of Talagrand's \cite[Section 16.10]{Tal14}.
Additionally, we note that this geometric perspective to concentration is crucial for several applications, such as hypergraph sparsification \cite{L23,JLS23} and quantum cryptography \cite{LMW24} (see \Cref{subsec:related_work} for more details).

\subsection{Our Contributions}
\label{subsec:contribution}

Our main contribution is a geometric approach, involving covering numbers and Gaussian process theory, to estimate the expected $\ell_p$ injective norm of a jointly Gaussian tensor $T$ as in \eqref{eq:Tasgaussianseries_TCS} for a full range of $r\geq 2$ and $ 2 \leq p <\infty$, where operator theoretic tools are unavailable. Our main result (\Cref{thm:masterthm_TCS} below) bounds this expected $\ell_p$ injective norm in terms of natural parameters associated with the random tensor $T$. 

\medskip
\noindent \textbf{Type-2 Constants of Banach Spaces.} Before formally defining these tensor parameters and stating our general result in \Cref{thm:masterthm_TCS}, we first present its implication to the type-2 constants of the Banach spaces corresponding to $\ell_p$ injective norms. It is considerably easier to state this specialized result and discuss its optimality.

\begin{defn}[Type-2 constant for $\ell_p$ injective norm of order $r$ tensors]\label{def:type2constant_TCS}
Given $r\geq 2$ an integer and $2 \leq p < \infty$, we define the type-2 constant  of the $\ell_p$ injective norm of order $r$ tensors  to be the smallest number $\CCC_{r,p}(d)$  such that for all positive integers $n$ and tensors $T_1, \ldots, T_n \in (\R^d)^{\otimes r}$, we have\footnote{The sum of squared norms on the right-hand side of \eqref{eq:type2def_TCS} have also appeared in \cite{Tom74,AW02} in the context of random matrices as a variance parameter, which is not as precise as the one given in~\eqref{eq:NCK_TCS} that gained attention after it appeared in matrix concentration inequalities~\cite{Oli10,Tro10}.}
\begin{equation}\label{eq:type2def_TCS}
    \E \Big\|\sum_{k=1}^n \eps_k T_k\Big\|_{\I_p}^2 \leq \CCC_{r,p}(d)^2 \cdot \sum_{k=1}^n \norminp{T_k}^2 ,
\end{equation}
where $\eps_1, \ldots, \eps_n$ are i.i.d. Rademacher variables, i.e., $ \eps_k = \pm 1$ with probability $\frac12$ each. 
\end{defn}

While we have been focusing on jointly Gaussian  tensors as in \eqref{eq:Tasgaussianseries_TCS}, a standard application of Jensen's inequality  on the magnitude of Gaussian random variables shows that moments of the norm of a Rademacher series are upper bounded by that of a Gaussian series. In particular, we have the following estimates. 

\begin{thm}[Type-2 constant bounds]\label{cor:boundtype2constant_TCS}
Let $r\geq 2$ an integer, $2 \leq p < \infty$, and $\CCC_{r,p}(d)$ the type-2 constant of the $\ell_p$ injective norm of order $r$ tensors as in \Cref{def:type2constant_TCS}. Then\footnote{The notation $f \lesssim_{r,p} g$ means that the inequality $ f \leq C_{r,p} g$ holds for some constant $C_{r,p}$ that only depends on $r$ and $p$ (see \Cref{subsec:notation_TCS}).}
\begin{equation}
d^{\frac12 - \frac1p}\cleq{r,p}  \CCC_{r,p}(d) 
\cleq{r,p} d^{\frac12- \frac1{\max\{p,2r\}}}\log d.
\end{equation}
\end{thm}

Note that for $p \geq 2r$, the upper and lower bounds in \Cref{cor:boundtype2constant_TCS} match up to a $\log d$ factor (in terms of dependence on $d$), and thus our results are nearly tight in this regime. 
In the regime of $p < 2r$, however, the bounds in \Cref{cor:boundtype2constant_TCS} differ by a $\poly(d)$ gap. 
For the correct polynomial dependence on $d$, both 
$d^{\frac12 - \frac1p}$ and $d^{\frac12- \frac1{\max\{p,r\}}}$ would be natural conjectures. Ruling either of them out seems to  be beyond reach of current techniques (see \Cref{subsection:ldcs} for a more detailed discussion).

\begin{rem} \label{rem:nckimprovedtype2_TCS}
    In the matrix case ($r=2$), it is possible to improve upon the upper bound in \Cref{cor:boundtype2constant_TCS} (in some regimes) by applying Hölder's inequality to NCK in \eqref{eq:NCK_TCS}:
    $$ \E \Big\|\sum_{k=1}^n g_k A_k \Big\|_{\I_p} \leq d^{1- \frac 2 p} \E \Big\|\sum_{k=1}^n g_k A_k\Big\|_2 \lesssim \sqrt{ \log d} \cdot d^{1- \frac 2 p}\sqrt{\sum_{k=1}^n \norminx{A_k}{2}^2} . $$
    Since $\norminx{A_k}{2} \leq \norminp{A_k}$, this inequality implies $\CCC_{2,p}(d) \lesssim \sqrt{ \log d} \cdot d^{1- \frac 2 p}$, which is better than the upper bound in \Cref{cor:boundtype2constant_TCS} when $2 \leq p \leq \frac{8}{3}$.
\end{rem}

\medskip
\noindent \textbf{Our Main Result.}
Our bounds are more accurately written in terms of tensor parameters that provide sharper control than the sum of squared injective norms in \eqref{eq:type2def_TCS}. 
To formally state our result in \Cref{thm:masterthm_TCS}, we first define these tensor parameters.

\begin{defn}[$\star_q$ product]\label{defn:introproduct_TCS}
    Let $A,B \in (\R^d)^{\otimes r}$ be tensors. For $0 \leq q \leq r$ we define a $d$-dimensional order $2r-2q$ tensor $A \star_q B$ with entries
    $$ (A \star_q B)_{i_1, \ldots, i_{2r-2q}} \coloneqq \sum_{j_1,\ldots, j_q =1}^d A_{i_1, \ldots, i_{r-q},j_1, \ldots, j_q} B_{j_1, \ldots, j_q,i_{r-q+1}, \ldots, i_{2r-2q}}.$$
\end{defn}
To illustrate this definition, in the matrix case ($r=2$) where $A,B \in \R^{d\times d}$ are symmetric matrices, $\star_q$ corresponds to familiar matrix operations: $A\star_0B=A \otimes B$ corresponds to tensor product (seen as an order $4$ tensor), $A\star_1B=AB$ to matrix multiplication, and $A\star_2B=\Tr(A^{\mathsf{T}}B)$ to the Hilbert-Schmidt inner product.

\begin{defn}[Tensor parameters]\label{defn:intovariances_TCS}
    Let $T \in (\R^d)^{\otimes r}$ be a symmetric random jointly Gaussian tensor and $2 \leq p < \infty $. For $0 \leq q \leq r$, we define the parameters
    $$ \pvar{q}^2 \coloneqq \norminp{ \E[ T \star_q T ]}.$$
\end{defn}

In the case of matrix operator norm ($p=r=2$), the parameter $\pvar{1}$ is the variance parameter in NCK \eqref{eq:NCK_TCS}, while $\pvar{0}$ also frequently appears in the literature\footnote{We define the weak variance in a different but equivalent form to that in~\cite{Tro15} (see \Cref{lemma:variances}).} and is often referred to as ``weak variance"~\cite{Tro15}. \Cref{lemma:variances} shows that for $p = r = 2$, 
\[ \evar{1}^2 = \Big\| \sum_{k=1}^n A_k^2 \Big\|_{\mathrm{op}} \qquad \textrm{and} \qquad \evar{0}^2 = \sup_{u,v \in \edball} \sum_{k=1}^n (u^\mathsf{T}A_k v)^2 .
\]
Formally stated, our main result provides the following upper bound:

\begin{thm}[Main Theorem]\label{thm:masterthm_TCS}
    Let $T \in (\R^d)^{\otimes r}$ be a symmetric random jointly Gaussian tensor with $\E[T] = 0$ as in \eqref{eq:Tasgaussianseries_TCS}, and $2 \leq p < \infty$. Then, 
    \begin{equation}\label{eq:intromainthm_TCS} 
    d^{\frac1p -\frac12} \E \norminp{T} \cleq{r,p} (\log d) \,  \pvar{1} +  \max_{2 \leq q \leq r} \pvar{q}^{\frac1q} \pvar{0}^{\frac{q-1}{q}},
    \end{equation}
    where the parameters on the right-hand side are those in \Cref{defn:introproduct_TCS} and~\ref{defn:intovariances_TCS}.
\end{thm}

We discuss applications and connections of this more general result to several other problems, such as tensor PCA, various models of random tensors, and locally decodable codes, in the subsequent \Cref{subsec:applications}.

\begin{par}   
    We finish this subsection with a quick remark about symmetric tensors. If one does not strive for an optimal dependence on the order $r$, the symmetry assumption in \Cref{thm:masterthm_TCS} is without loss of generality via symmetric embeddings of tensors (see \Cref{subsection:symbed}). 
\end{par}

\subsection{Applications and Connections}
\label{subsec:applications}

\noindent \textbf{Tensor PCA.} The {\em tensor principle component analysis} (tensor PCA) problem was introduced by Montanari and Richard \cite{RM14} as a model for studying statistical inference under observations of higher order interactions among data elements. In this model (also known as the {\em spiked tensor} model), one is given $\lambda >0$, a signal $v \in \{\pm 1\}^d$, and noisy observations of the $r$-order tensor 
\[
Y(\lambda) = \lambda v^{\otimes r} + T, 
\]
where $T \in (\R^d)^{\otimes r}$ is a random noise tensor with i.i.d. standard Gaussian entries. 
The key question in this context is to understanding for which signal-to-noise ratio $\lambda$ can one {\em detect} the presence of the signal $v$, and reliably {\em recover} it. 

In the case where the noisy tensor $Y(\lambda)$ is fully observed, it is known that information theoretically, one can detect and recover the signal $v$ when $\lambda \gg d^{(1-r)/2}$ \cite{RM14,LML+17}. The sum-of-squares (SoS) hierarchy or Kikuchy hierarchy provide efficient algorithms for detection and recovery when $\lambda \gg d^{-p/4}$ \cite{HSS15,WEAM19,KMW24}, and SoS lower bounds provide evidence that no polynomial time algorithm can do better \cite{HSS15,HKP+17}. Asymptotically optimal information-theoretic thresholds for detection and recovery are also shown when the signal $v$ is drawn from various priors \cite{PWB20}. 

We study the tensor PCA problem in the more general censored (or partial information) model, where potentially only part of the entries of the noisy tensor $Y(\lambda)$ is available. This type of partial information models have been well-studied for matrices (corresponding to graphs), where the statistical threshold is characterized by properties of the underlying graphs (e.g., \cite{ABBS14,C15,BCSvH24}). To the best of our knowledge, there is no non-trivial bound known in the censored model for tensor PCA, which might be partially due to a lack of understanding of random tensors beyond the i.i.d. entry or rank-$1$ cases.

As an application, our result in \Cref{thm:masterthm_TCS} leads to the first non-trivial information-theoretic threshold for detection in the setting of censored tensor PCA, in terms of natural quantities associated with the hypergraph corresponding to the observed entries of $Y(\lambda)$. 
This result also generalizes the information-theoretically optimal threshold for the full information setting \cite{RM14,LML+17} (up to a logarithmic factor).
We postpone the formal definition of the censored tensor PCA problem to \Cref{subsec:ind_ent_tensor_PCA} and the statement of this result to \Cref{thm:lambdastatpca}.

\medskip
\noindent \textbf{Nonhomogenous Independent Entry Tensors.} As a key intermediate step towards proving our result for censored tensor PCA in \Cref{thm:lambdastatpca}, we obtain improved bounds for the $\ell_2$ injective norm of a random tensor with nonhomogeneous independent entries\footnote{This means the entries of the random tensor are independent (up to symmetry) Gaussian random variables with arbitrary variances.} (see \Cref{thm:indepentryapplication} and \Cref{cor:hypergraphindepentry} for the formal statements). 

The nonhomogeneous model, a special case of the jointly Gaussian model in \eqref{eq:Tasgaussianseries_TCS} and its Rademacher version, has received significant attention in the modern development of random matrix theory \cite{S00,Lat05,BvH16,vH17b,LvHY18}. It is not only among the first natural random matrix models (beyond classical Wigner ensembles) where sharper bounds than NCK \eqref{eq:NCK_TCS} have been attained, but also a testbed where geometric approaches have been successfully applied \cite{vH17,Lat24,GHLP17,APSS24}. 

In contrast, to the best of our knowledge, there has not been any notable developments of the nonhomogeneous independent model for tensors (see \Cref{subsection:refineindepentry} for further discussions). Despite not relying on the independence structure of the entries, our result in \Cref{thm:masterthm_TCS} implies the first non-trivial bound for this natural model (\Cref{thm:indepentryapplication} and \Cref{cor:hypergraphindepentry}).  It is an intriguing open problem to further improve on these bounds by taking advantage of the independence of entries in the nonhomogeneous independent entry model (which might also lead to better bounds for censored tensor PCA). 

\medskip
\noindent \textbf{Structured Random Matrices/Tensors.} 
Beyond the independent entry model, our result in \Cref{thm:masterthm_TCS} also implies better bounds for more structured models, even in the case of random matrices. We demonstrate such an application in \Cref{subsec:matching_mat}, where each $T_i$ in \eqref{eq:Tasgaussianseries_TCS} corresponds to the adjacency matrix of a matching $M_i$. We show that our bound in \Cref{thm:masterthm_TCS} is superior than previous bounds when $M_i$ have certain structures.

\medskip
\noindent \textbf{Type-$2$ Constants of Banach Spaces.} The (Rademacher) type and cotype and their corresponding constants are fundamental quantities associated with Banach spaces that capture a notion of orthogonality, e.g., \cite{AK06,LT13}. They have been widely studied in the literature for various Banach spaces, and are connected to applications such as locally decodable codes \cite{BNR12}. As already discussed in \Cref{subsec:contribution}, our main result in \Cref{thm:masterthm_TCS} implies better bounds for the type-2 constants of the Banach spaces corresponding to $\ell_p$ injective norms of tensors, which are tight up to constants when $p \geq 2r$.

\medskip
\noindent \textbf{Locally Decodable Codes.} Locally decodable codes (LDC) are ones where each message bit can be decoded with non-trivial probability by (randomly) accessing only a small number $q$ of bits of the codeword. They are widely studied and are closely connected to many other central questions in computer science \cite{Yek12,Gop18} (also see \Cref{subsection:ldcs}).  
A natural question for LDCs is the smallest achievable blocklength $n$ for different values of $q$ (in terms of the message length $k$). Although the case of $q = 2$ is quite well-understood, there is a significant gap between the current upper and lower bounds when $q \geq 3$.  

While the bound in \Cref{thm:masterthm_TCS} currently falls short of providing improved bound for LDCs, we show that a natural strengthening of \Cref{thm:masterthm_TCS} (as stated in \Cref{conj:actualtype2constant}) would imply better upper bound on the blocklength $n$ for LDC with an odd number of queries $q > 3$. Notably, such an improvement does not follow from the best possible bound on the $\ell_2$ injective norm for tensor concentration inequalities; it crucially relies on considering the $\ell_q$ injective norms of order-$q$ tensors. We postpone a more detailed discussion of this connection to LDCs to \Cref{subsection:ldcs}.

\subsection{Further Related Work}
\label{subsec:related_work}

\medskip
\noindent \textbf{Applying Matrix Concentration Inequalities to Tensor Applications.} In applications where (random) tensors are natural objects of study, matrix concentration inequalities can still be used by first ``flattening out'' the random tensor into a matrix that captures certain spectral properties. For example, matrix concentration inequalities have been successfully applied to the higher-dimensional Kikuchy matrix (corresponding to tensor problems) in applications such as tensor PCA \cite{WEAM19}, refuting semi-random and smoothed  CSPs \cite{GKM22}, hypergraph Moore bound \cite{HKM23}, lower bound for 3-LDCs and LCCs \cite{AGKM23,KM24}. Other applications of tensor flattening approaches include \cite{GM15,HSSS16,BDNY24}. The success of these approaches often rely on additional structural properties of the underlying problem. 

\medskip
\noindent \textbf{Geometric Perspectives of Concentration Inequalities.} Besides being interesting and fundamental mathematical questions in their own right, developing geometric  approaches for matrix and tensor concentration inequalities have strong motivations from applications where either the objects are not matrices (e.g., tensors, norms, etc.), or when the norms are not the operator norm. For example, they are useful for hypergraph sparsification \cite{L23,JLS23}, sparsifying sums of norms \cite{JLLS23}, sparsifying generalized linear models \cite{JLLS24}, breaking quantum cryptography \cite{LMW24}, and studying nonhomogeneous independent entry random matrices \cite{vH17,GHLP17,Lat24,APSS24}.

\subsection{Notation and Terminology}
\label{subsec:notation_TCS}

We use $\|\cdot\|_p$ to denote the $\ell_p$ norm in $\R^d$ and $\pdball$ its unit ball. 
For $u_1, \ldots , u_r \in \R^d$, the tensor $u_1 \otimes \cdots \otimes u_r \in (\R^{d})^{\otimes r}$ is given by $(u_1 \otimes \cdots \otimes u_r)_{i_1, \ldots, i_r} = \prod_{q=1}^r (u_q)_{i_q}$; in particular, denote $v^{\otimes r} := v \otimes \cdots \otimes v$.
For tensors $A, B \in (\R^d)^{\otimes r}$, define 
\[
\langle A, B \rangle \coloneqq \sum_{i_1, \ldots, i_r=1}^d A_{i_1, \ldots, i_r} B_{i_1, \ldots, i_r} = A \star_r B ,
\]
where we recall $\star_q$ from \Cref{defn:introproduct_TCS}. We use $A \odot B \in (\R^d)^{\otimes r}$ to denote the entry-wise product.
We will also see tensors as multilinear maps on $\R^d$: 
\[
A[u_1, \ldots, u_r] \coloneqq \sum_{i_1,\ldots, i_r=1 }^d A_{i_1, \ldots, i_r}(u_1)_{i_1} \cdots (u_r)_{i_r} = \langle A, u_1 \otimes \cdots \otimes u_r \rangle 
\]
This allows contracting tensors with vectors. For $v \in \R^d$, we recursively define $Av^{\otimes 0} \coloneqq A$ and $Av^{\otimes q} \coloneqq (Av^{\otimes {q-1}})[v, \cdot, \ldots, \cdot]$ for $q \leq r$.
Note that $Av^{\otimes q}$ is an order $r-q$ tensor.
The Frobenius norm is given by $\normf{A} \coloneqq \sqrt{\langle A, A \rangle}$. 

We denote $[r] := \{ 1, \ldots, r \}$, and the set of all permutations on $[r]$ by $S_r$. A tensor $A$ is called \emph{symmetric} if for all permutations $\tau \in S_r$ and all $1 \leq i_1, \ldots, i_r \leq d$, we have $T_{i_1, \ldots i_r} = T_{i_{\tau(1)}, \ldots i_{\tau(r)}} $. 
The symbols $\lesssim, \lesssim_r, \lesssim_{r,p}$ are inequalities that hold up to a factor that are respectively universal constants, constants depending on $r$, constants depending on $r$ and $p$. If both $\lesssim$ and $\gtrsim$ hold, we also use the symbols $\asymp, \asymp_r, \asymp_{r,p}$ in a similar fashion.

\subsection{Technical Overview}
\label{sec:technique_TCS}

In this section, we give a brief overview of the techniques used to derive our results. We focus on the proof of \Cref{thm:masterthm_TCS}, from which \Cref{cor:boundtype2constant_TCS} can be deduced by upper bounding the tensor parameters via comparison inequalities between different tensor norms (see \Cref{thm:boundedmultilinear}). 
A formal presentation of our proofs can be found in \Cref{sec:proof}.

\medskip
\noindent \textbf{The Gaussian Process View.} 
For a symmetric jointly Gaussian tensor $T \in (\R^d)^{\otimes r}$ as in \eqref{eq:Tasgaussianseries_TCS}, its $\ell_p$ injective norm  can be viewed as the supremum of a gaussian process\footnote{Formally, one needs to use the equivalence between symmetric and asymmetric versions of the $\ell_p$ injective norms given in \Cref{prop:symnormequiv}, which only loses a constant factor depending on $r$.}
\[
\|T\|_{\I_p} \approx \E \sup_{u \in \mathbb{B}_p^d} g_u , \quad \text{where } g_u := \langle T, u^{\otimes r} \rangle .
\]
As we do not aim for optimal logarithmic factors, we can apply Dudley's entropy integral (see \Cref{thm:entropyintegral}) to bound the expected supremum of this Gaussian process as
    \begin{align} \label{eq:dudley_int_technique}
    \E \sup_{u \in \pdball} g_u \lesssim \int_0^\infty \sqrt{\log \mathcal{N} (\pdball,\md, \eps)} \,d\eps ,
    \end{align}
where $\mathcal{N} (\pdball,\md, \eps)$ denotes the {\em covering number} of $\pdball$ at scale $\varepsilon$ w.r.t. to the (semi-)metric $\md(u,v) := \E [(g_u - g_v)^2]^{\frac12}$, i.e. it is the smallest cardinality of an $\eps$-covering of $\pdball$ in $\md$. We estimate this covering numbers using a geometric argument in the following. 

\medskip
\noindent \textbf{Warm-Up: Operator Norm for Matrices.}  We first present a covering number bound in the simplified setting where (1) $p = r = 2$, i.e., the operator norm of matrices, (2) $n = d$, and (3) each $\|T_k\|_{\op} \leq 1$. 
In this special case, the distance function $\md$ is given by
\begin{align} \label{eq:warm_up_dist}
\md(u, v)^2 = \sum_{k=1}^d \big(\langle T_k, uu^\top \rangle - \langle T_k, vv^\top \rangle \big)^2 . 
\end{align}
The proof in this case already contains many of the key ideas behind the proof of \Cref{thm:masterthm_TCS}, and the quantities are more explicit. 
We may further assume $T_k$'s are {\em diagonal-free}, i.e. $\diag(T_k) = 0$, by replacing $T_k$ with its Hermitian dilation (see \Cref{subsection:symbed}).

For small $\varepsilon \in (0,1]$, the standard estimate $\NNN(B_2^d, \md, \varepsilon) \leq (3d/\varepsilon)^d$ already suffices  (see \Cref{cor:pdlogcovernumberspnorm}), so we focus on $\varepsilon > 1$. In this regime, we prove the bound 
\begin{align} \label{eq:warm_up_cov_bound}
    \NNN(B_2^d, \md, \varepsilon) \lesssim e^{O(d / \varepsilon)} .
\end{align}
Once \eqref{eq:warm_up_cov_bound} is established, plugging it into \eqref{eq:dudley_int_technique} gives\footnote{Here, the diameter of $B_2^d$ in metric $\md$ is $\sqrt{d}$, so the integration in \eqref{eq:dudley_int_technique} can be truncated at $\sqrt{d}$.} a bound of $\E\|T\|_{\op} \lesssim d^{3/4}$, which is weaker than the $O(\sqrt{d \log d})$ bound achieved by non-commutative Khintchine in \eqref{eq:NCK_TCS}. We remark that up to a logarithmic factor, the NCK bound would be recovered if one could prove the better covering number bound of $\NNN(B_2^d, \md, \varepsilon) \lesssim e^{O(d / \varepsilon^2)}$ instead of \eqref{eq:warm_up_cov_bound}. This might also provide a proof strategy that recovers NCK up to logarithmic factors, getting close to answering Talagrand's question in \cite{Tal14}, and is left as an open problem.

In the following, we present the key ideas behind the proof of \eqref{eq:warm_up_cov_bound}. To bound the covering numbers $\NNN(B_2^d, \md, \varepsilon)$, it suffices to bound the size of a maximal set $D = \{x_1, \cdots, x_N\}$ that is $\varepsilon$-separated in $\md$, i.e. $\md(x_i, x_j) > \varepsilon$. In particular, we show that 
\begin{align} \label{eq:warm_up_pack_bound}
    N \lesssim e^{O(d / \varepsilon)} ,
\end{align}
which immediately implies \eqref{eq:warm_up_cov_bound} by a standard packing to covering argument.

\medskip
\noindent \textbf{Volume Counting: A First Attempt.} We establish \eqref{eq:warm_up_pack_bound} via a volume counting argument. Specifically, note that the regions $B_{\md}(x_i, \frac{\varepsilon}{2}) := \{x \in \R^d: \md(x, x_i) < \varepsilon/2\}$ are mutually disjoint, since $d(x_i, x_j) > \varepsilon$. Consequently, it makes sense to try to define regions $G_i \subseteq B_{\md}(x_i, \frac{\varepsilon}{2})$ which satisfy the following conditions: 
\begin{enumerate}
    \item [(a)] The volumes of $G_i$'s are ``large''.
    \item [(b)] $G_i$ are all enclosed in a set $K$ whose volume is ``small''. 
\end{enumerate}
Once these conditions are satisfied, as $G_i$'s are disjoint, one immediately obtains the bound 
\[
N \leq \frac{\vol \big(\cup_{i \in [N]} G_i \big)}{\min_{i \in [N]} \vol(G_i)} \leq \frac{\vol(K)}{\min_{i \in [N]} \vol(G_i)} .
\]
So the main question is: how should we define $G_i$'s satisfying both conditions (a) and (b)?  

As a first attempt, one might want to define $G_i$ to be $B_{\md}(x_i, \frac{\varepsilon}{2}) \cap B_2^d$, i.e. choose $K = B_2^d$ in (b). This choice is rather natural since all points in $D$ lie inside $B_2^d$. Unfortunately, this instinctive idea does not work for the following reason. When $T_1 = \cdots = T_d$ (with $\|T_i\|_i=1$) and $x_j$ is the maximum eigenvector of the $T_i$'s, we have
\[
\md(x_j, u)^2 = \sum_{k=1}^d \big(\langle T_k, x_jx_j^\top \rangle - \langle T_k, uu^\top \rangle \big)^2 = d (1 - \langle T_k, uu^\top \rangle)^2.  
\]
For the distance to be smaller than $\varepsilon/2$, one must have $\langle T_k, uu^\top \rangle \geq 1 - \varepsilon/2d$, which means that  $u$ has to lie very close to $x_j$. Consequently, the $G_j$ defined this way have volume too small (aka failing condition (a)) to obtain a non-trivial bound on $N$ in \eqref{eq:warm_up_pack_bound}. 

\medskip
\noindent \textbf{Volume Counting Outside the Euclidean Ball.} Surprisingly, we are able to bypass the aforementioned issue by carrying out the volume counting argument {\em far outside} the Euclidean ball $B_2^d$, despite that we are bounding the covering number inside $B_2^d$.

In particular, we define the regions $\widetilde{G}_i := B_{\md}(x_i, \frac{\varepsilon}{2}) \cap (x_i + R B_2^d)$, where we choose $R > 1$ (to be fixed later) so that $\vol(\widetilde{G}_i) \geq \frac{3}{4} \vol(R B_2^d)$. This guarantees the  satisfaction of condition (a). For condition (b), there is of course the trivial bound of $\widetilde{G}_i \subset (R+1) B_2^d$. 
Nonetheless, this trivial bound is not strong enough to establish \eqref{eq:warm_up_cov_bound}. 

To achieve the desired bound, it turns out we can further strengthen (b) by ``slicing the ball $x_i + R  B_2^d$ and define $G_i := \widetilde{G}_i \cap \{x \in x_i + R  B_2^d: \langle x - x_i, x_i \rangle > 0\}$. This additional linear constraint implies that $G_i \subseteq \sqrt{R^2 + 1} \cdot B_2^d$ (which is considerably smaller than $(R+1) B_2^d$), while we still have $\vol(G_i) \geq \frac{1}{4} \vol(R B_2^d)$. It follows that 
\begin{align} \label{eq:warm_up_vol_count}
N \leq \frac{\vol(\sqrt{R^2 + 1} B_2^d)}{\frac{1}{4} \cdot \vol(R B_2^d)} \lesssim \Big( \frac{R^2 + 1}{R^2} \Big)^{d/2} \leq e^{d/2R^2} .
\end{align}
So it remains to decide how large  $R$ can be set so that $\vol(\widetilde{G}_i) \geq \frac{3}{4} \vol(R B_2^d)$ holds.  

\medskip
\noindent \textbf{Determining the Radius.} To this end, we estimate the expected distance of a uniformly random vector $\rv{y} \in x_i + R B_2^d$ from $x_i$ as
\begin{align*}
\E [\md(x_i,\rv{y})^2 ] 
& = \sum_{k=1}^d \E[(\langle T_k, \rv{y} \rv{y}^\top \rangle - \langle T_k, x_i x_i^\top \rangle)^2] 
\lesssim \sum_{k=1}^d \E\Big[R^2 \langle T_k, x_i \rv{b}^\top \rangle^2 +  R^4  \langle T_k, \rv{b} \rv{b}^\top \rangle^2 \Big]  ,
\end{align*}
where $\rv{b}$ is uniformly random from $B_2^d$. We skip through the precise computation (see \Cref{lemma:expectedpnormcontraction} and \Cref{cor:expectedpdistance} for more general bounds), but it is possible to bound 
\[
\E[\langle T_k, x_i \rv{b}^\top \rangle^2], \ \E[\langle T_k, \rv{b} \rv{b}^\top \rangle^2] \lesssim \frac{1}{d} . 
\]
We remark that it is crucial that $\rv{b}$ has sufficient ``symmetry'' for this computation to go through (as $T_k$'s can be quite arbitrary), and this is another technical reason why we have to work outside the unit ball $B_2^d$. 
Consequently, we have $\E [\md(x_i,\rv{y})] \lesssim R^2$. Thus choosing $R = \sqrt{\varepsilon}/C$ for a large enough constant $C$, Markov's inequality suggests that $\vol(\widetilde{G}_i) \geq \frac{3}{4} \vol(R B_2^d)$. Plug this choice of $R$ into \eqref{eq:warm_up_vol_count} immediately implies our desired covering number bound in \eqref{eq:warm_up_pack_bound} and \eqref{eq:warm_up_cov_bound}. 

Note that as $\varepsilon$ approaches the diameter $\sqrt{d}$ (i.e., the Dudley integration upper limit), the choice of $R$ grows to $d^{1/4}/C$ which is substantially larger than the unit ball $B_2^d$ that we aim to bound the covering numbers for in the first place. It is quite surprising that in such enormous space outside of the unit ball, sufficient metric information between the points in $D$ gets preserved for one to still deduce a meaningful bound on its cardinality!

\medskip
\noindent \textbf{Generalization to Tensors and $p \geq 2$.} The argument above provides a systematic avenue for estimating $\ell_p$ injective norms of tensors. This approach is quite mechanical -- all one has to do is to (1) upper bound the expected distance $\E[\md(x_i, \rv{y})]$ for a uniform $\rv{y} \sim x_i + R B_p^d$ to determine the largest possible radius $R$ for (a), and (2) find a way to ``slice'' $B_p^d$ to obtain a better volume bound in (b) than the trivial $\vol((R+1) B_p^d)$. 

We show how to do (1)  in \Cref{lemma:expectedpnormcontraction} and \Cref{cor:expectedpdistance} through a direct computation, and (2) in \Cref{cor:halvinglpball} via a consequence of the $2$-uniform convexity for $\ell_p$ norms in \Cref{thm:weakparallelogram}. 
Combine these bounds with our general machinery leads to \Cref{thm:masterthm_TCS}. The full details of our proofs are postponed to \Cref{sec:proof}.

\section*{Acknowledgements}
The authors would like to thank March Boedihardjo, Tim Kunisky, Antoine Maillard, Shahar Mendelson, Petar Nizić-Nikolac and Victor Reis for helpful discussions and insights. A special thanks to Ramon van Handel for bringing our attention to~\cite{Lat06}.

\bibliographystyle{alpha}
\bibliography{bib.bib}

\end{document}